\documentclass[12pt,reqno]{amsart}
\usepackage{amsmath,amssymb,extarrows}
\usepackage{url}
\usepackage{tikz,enumerate}
\usepackage{diagbox}
\usepackage{appendix}
\usepackage{epic}

\usepackage{float} 

\usepackage{cite}

\usepackage{hyperref}
\usepackage{array}

\usepackage{booktabs}

\allowdisplaybreaks[4]

\newtheorem{theorem}{Theorem}
\newtheorem{corollary}[theorem]{Corollary}

\newtheorem{lemma}[theorem]{Lemma}

\theoremstyle{definition}

\theoremstyle{remark}
\newtheorem{rem}{Remark}
\numberwithin{equation}{section}
\numberwithin{theorem}{section}
\numberwithin{defn}{section}

\begin{document}
\title[Zagier's rank three examples for Nahm's problem]
 {Explicit Forms and Proofs of Zagier's Rank Three Examples for Nahm's Problem}

\author{Liuquan Wang}
\address{School of Mathematics and Statistics, Wuhan University, Wuhan 430072, Hubei, People's Republic of China}
\email{wanglq@whu.edu.cn;mathlqwang@163.com}



\begin{abstract}
Let $r\geq 1$ be a positive integer, $A$ a real positive semi-definite symmetric $r\times r$ rational matrix, $B$ a rational vector of length $r$, and $C$ a rational scalar. Nahm's problem is to find all triples $(A,B,C)$ such that the $r$-fold $q$-hypergeometric series
$$f_{A,B,C}(q):=\sum_{n=(n_1,\dots,n_r)^\mathrm{T}\in (\mathbb{Z}_{\geq 0})^r} \frac{q^{\frac{1}{2}n^\mathrm{T} An+n^\mathrm{T} B+C}}{(q;q)_{n_1}\cdots (q;q)_{n_r}}$$
becomes a modular form, and we call such $(A,B,C)$ a modular triple. When the rank $r=3$, after extensive computer searches, Zagier provided twelve sets of conjectural modular triples and proved three of them. We prove a number of Rogers-Ramanujan type identities involving triple sums. These identities give modular form representations for and thereby verify all of Zagier's rank three examples. In particular, we prove a conjectural identity of Zagier.
\end{abstract}

\maketitle

\section{Introduction}\label{sec-intro}
In his 1894 paper, L.J.\ Rogers \cite{Rogers1894} discovered many sum-product $q$-series identities. Among other results,
he proved that \cite[pp.\ 328,330,331]{Rogers1894}
\begin{align}
&\sum_{n=0}^\infty \frac{q^{n^2}}{(q;q)_n}=\frac{1}{(q,q^4;q^5)_\infty}, \label{RR-1} \\
&\sum_{n=0}^\infty \frac{q^{n^2+n}}{(q;q)_n}=\frac{1}{(q^2,q^3;q^5)_\infty}, \label{RR-2} \\
&\sum_{n=0}^\infty \frac{q^{n^2}}{(q^4;q^4)_n} =\frac{1}{(-q^2;q^2)_\infty (q,q^4;q^{5})_\infty}, \label{Slater20} \\
&\sum_{n=0}^\infty \frac{q^{n(n+2)}}{(q^4;q^4)_n}=\frac{1}{(-q^2;q^2)_\infty (q^2,q^{3};q^{5})_\infty}. \label{Slater16}
\end{align}
Here and throughout this paper, we always assume $|q|<1$ for convergence and use the following $q$-series notation:
\begin{align}
(a;q)_0:=1, \quad (a;q)_n:=\prod\limits_{k=0}^{n-1}(1-aq^k), \quad (a;q)_\infty :=\prod\limits_{k=0}^\infty (1-aq^k),  \\
(a_1,\cdots,a_m;q)_n:=(a_1;q)_n\cdots (a_m;q)_n, \quad n\in \mathbb{N}\cup \{\infty\}.
\end{align}

Rogers' works on this kind of sum-product identities were quite neglected for some time. According to G.H.\ Hardy \cite{Hardy}, S. Ramanujan rediscovered the identities \eqref{RR-1}--\eqref{RR-2} before 1913 and he came across Rogers' paper \cite{Rogers1894} in 1917. Thereafter, the identities \eqref{RR-1} and \eqref{RR-2} were usually referred as the Rogers-Ramanujan identities. Around 1917, I. Schur \cite{Schur} independently rediscovered these two identities and gave combinatorial interpretations to them.

The Rogers-Ramanujan identities play important roles in different branches of mathematics and physics. In combinatorics, it has interesting partition interpretations and stimulate a number of researches on finding similar partition identities. It is closely related to vertex operator algebras, representation theory of Lie algebras and knot theory. Moreover, the functions in these identities are modular forms, and also appear in solutions to the hard hexagon model in statistical mechanics.

Ramanujan's rediscovery of the identities \eqref{RR-1}--\eqref{RR-2} attracted people's attention to Rogers' works. Interests on finding similar identities, which are usually called as Rogers-Ramanujan type identities have lasted for more than one hundred years and are still continuing. One of the famous works on this topic is L.J.\ Slater's list \cite{Slater}, which provides 130 such identities including some known results. For example, the identities \eqref{Slater20} and \eqref{Slater16} appear as equations (20) and (16) in \cite{Slater}, respectively. An elaborate introduction to Rogers-Ramanujan type identities can be found in A.V.\ Sills' book \cite{Sills-book}.

As mentioned before, the functions in \eqref{RR-1} and \eqref{RR-2} are both modular forms, which is clear from the product sides but not clear from the sum sides. An important question in the theory of $q$-series and modular forms is to judge what kind of basic hypergeometric series are modular forms. This question has not been completely answered yet. In a series of works, W.\ Nahm \cite{Nahm1994,Nahmconf,Nahm2007} considered the series
$$f_{A,B,C}(q):=\sum_{n=(n_1,\dots,n_r)^\mathrm{T} \in (\mathbb{Z}_{\geq 0})^r} \frac{q^{\frac{1}{2}n^\mathrm{T} An+n^\mathrm{T} B+C}}{(q;q)_{n_1}\cdots (q;q)_{n_r}},$$
where $r\geq 1$ is a positive integer, $A$ is a real positive definite symmetric $r\times r$ matrix, $B$ is a vector of length $r$, and $C$ is a scalar. Nahm posed the problem to describe all such $A,B$ and $C$ with rational entries for which $f_{A,B,C}(q)$ is a modular form. Following notions in the literature, we shall call such $(A,B,C)$ as a modular triple, and call $A$ the matrix part, $B$ the vector part and $C$ the scalar part of it. It is worth mentioning that Nahm's motivation of this problem comes from physics and the modular forms $f_{A,B,C}(q)$ are expected to be characters of rational conformal field theories. 

Nahm \cite{Nahm2007} made a conjecture which provides sufficient and necessary conditions on the matrix part of a modular triple. The conjecture is formulated in terms of the Bloch group and a system of polynomial equations induced by the matrix part. We refer the reader to D.\ Zagier's paper \cite[p.\ 43]{Zagier} for precise statement of this conjecture.

When the rank $r=1$, the identities \eqref{RR-1}--\eqref{Slater16} showed that
\begin{align}\label{rank1-Rogers}
(A,B,C)=(2,0,-1/60), ~~ (2,1,11/60), ~~ (1/2,0,-1/40), ~~ (1/2,1/2,1/40)
\end{align}
are all modular triples. Zagier \cite{Zagier} studied Nahm's problem and found many possible modular triples. In particular, when the rank $r=1$, Zagier confirmed Nahm's conjecture and proved that there are exactly seven modular triples. Besides the four aforementioned triples, the other three are
\begin{align}\label{rank1}
\left(1,0,-1/48\right), \quad \left(1,1/2,1/24\right), \quad \left(1,-1/2,1/24\right),
\end{align}
which is easily justified by Euler's identities (see \eqref{Euler}).

Nahm's problem becomes more difficult when the rank is larger. For $r=2,3$, Zagier \cite{Zagier} did extensive computer searches on possible modular triples. In the rank two case, similar searches have also been done by M.\ Terhoeven \cite{Terhoeven}. When the rank $r=2$, Zagier found eleven sets of possible modular triples \cite[Table 2]{Zagier}.  Specifically speaking, Zagier found several values of $B$ and $C$ such that $f_{A,B,C}(q)$ is or appears to be modular for $A$ being
\begin{align*}
\begin{pmatrix} \alpha & 1-\alpha \\ 1-\alpha & \alpha \end{pmatrix}, \quad \begin{pmatrix} 2 & 1 \\ 1 & 1 \end{pmatrix}, \quad
\begin{pmatrix} 4 & 1 \\ 1 & 1 \end{pmatrix}, \quad \begin{pmatrix} 4 & 2 \\ 2 & 2 \end{pmatrix}, \quad \begin{pmatrix} 2 & 1 \\ 1 & 3/2 \end{pmatrix}, \quad \begin{pmatrix}  4/3 & 2/3 \\ 2/3 & 4/3 \end{pmatrix}
\end{align*}
and their inverses. When the rank $r=3$,  Zagier found twelve possible modular triples \cite[Table 3]{Zagier}.

However, Zagier did not verify all of his (conjectural) examples. He stated explicit identities which reveal the modularity of $f_{A,B,C}(q)$ only in few cases. In the rank two case, he proved his example for $A=\left(\begin{smallmatrix} \alpha & 1-\alpha \\  1-\alpha & \alpha \end{smallmatrix}\right)$. Let $\lfloor x\rfloor $ denote the integer part of $x$. Zagier also stated the following conjectural identities:
\begin{align}
f_{\left(\begin{smallmatrix} 4 & 1 \\ 1 & 1 \end{smallmatrix}\right), \left(\begin{smallmatrix} 0 \\ 1/2\end{smallmatrix}\right), \frac{1}{120}}(q)=\frac{1}{(q;q)_\infty}\sum_{n\equiv 1 \pmod{10}} (-1)^{\lfloor n/10\rfloor}q^{n^2/20-1/24},  \label{Zagier-exam4-1}\\
f_{\left(\begin{smallmatrix} 4 & 1 \\ 1 & 1 \end{smallmatrix}\right), \left(\begin{smallmatrix}2 \\ 1/2 \end{smallmatrix}\right),  \frac{49}{120}}(q)=\frac{1}{(q;q)_\infty}\sum_{n\equiv 3\pmod{10}} (-1)^{\lfloor n/10\rfloor }q^{n^2/20-1/24}. \label{Zagier-exam4-2}
\end{align}
Assuming the truth of these two identities, this justifies his example for $A=\left(\begin{smallmatrix} 4 & 1 \\ 1 & 1 \end{smallmatrix}\right)$. When the rank $r=3$, Zagier only proved the first three examples, which correspond to
\begin{align}\label{eq-rankthree}
A=\begin{pmatrix} \alpha h^2+A_1 & \alpha & -\alpha h \\ \alpha h & \alpha & 1-\alpha \\
-\alpha h & 1-\alpha & \alpha \end{pmatrix}, \quad A_1\in \{1/2,1,2\}.
\end{align}
For the example corresponding to
\begin{align}
A=\begin{pmatrix} 2 &1 &1 \\ 1 & 2 & 0 \\ 1 & 0 & 2 \end{pmatrix}, \quad
& B =\begin{pmatrix} 0 \\ \nu \\ -\nu \end{pmatrix}, \quad C=\nu^2-\frac{1}{24}, \quad \nu \in \mathbb{Q},
\end{align}
Zagier stated (conjecturally) that \cite[Eq.\ (33)]{Zagier}
\begin{align}\label{eq-Zagier-conj}
f_{A,B,C}(q)=\frac{1}{(q;q)_\infty} \sum_{n\in \mathbb{Z}+\nu} q^{n^2-1/24}.
\end{align}
He mentioned that this identity is ``checked only numerically''.

After the work of Zagier \cite{Zagier}, the rank two case has been discussed in several works. For instance, M.\ Vlasenko and S.\ Zwegers \cite{VZ} found all modular triples $(A,B,C)$ for $A$ being $\left(\begin{smallmatrix} a & \lambda-a \\ \lambda-a & a \end{smallmatrix}\right)$ with $a\in \mathbb{Q}$ and $\lambda \in \{\frac{1}{2},1,2\}$. One of the remarkable outcomes of their results is that they provided two counterexamples to Nahm's conjecture. Specifically, they found that when $A$ is
$\left(\begin{smallmatrix} 3/2 & 1/2 \\ 1/2 & 3/2 \end{smallmatrix}\right)$ or $\left(\begin{smallmatrix} 3/4 & -1/4 \\ -1/4 & 3/4  \end{smallmatrix}\right)$, which do not satisfy Nahm's criterion, there are $B$ and $C$ such that $(A,B,C)$ are modular triples.
Besides, Vlasenko and Zwegers \cite{VZ} proved Zagier's example for $A=\left(\begin{smallmatrix} 1 & -1/2 \\ -1/2 & 1 \end{smallmatrix}\right)$. The also stated some conjectural identities for $A=\left(\begin{smallmatrix} 4/3 & 2/3 \\ 2/3 & 4/3\end{smallmatrix}\right)$, which were later justified by I. Cherednik and B. Feigin \cite{Feigin} via nilpotent double affine Hecke algebras. C. Calinescu, A. Milas and M. Penn \cite{CMP} proved some identities which verify Zagier's example for $A=\left(\begin{smallmatrix} 1 & -1  \\ -1 &2 \end{smallmatrix}\right)$. In his thesis,  C.-H. Lee \cite{LeeThesis} partially confirmed Zagier's examples for $A=\left(\begin{smallmatrix} 2 & 1 \\ 1 & 1 \end{smallmatrix}\right)$.  In a recent work of the author \cite{Wang2022}, together with several known cases in the literature, we verified ten of Zagier's rank two examples and stated conjectural identities for the remaining one.

While the rank two examples have been verified in several works, in contrast, Zagier's rank three examples have not been discussed much. For convenience, let us label the rank three examples in Zagier's list from 1 to 12 according to their order in \cite[Table 3]{Zagier}. It should be noted that the matrix parts in Examples 4 and 5 are only positive semi-definite. For convenience, we shall relax the restriction on the matrix part $A$ in the definition of modular triples by allowing $A$ to be positive semi-definite. As mentioned before, Zagier verified Examples 1-3 (see \eqref{eq-rankthree}). Moreover, Example 9 corresponds to special cases of the Andrews-Gordon identity (see \eqref{AG}). Explicit forms for other examples were not stated and their correctness were not known before this work.

The aim of this paper is to present explicit identities and give complete proofs for all of Zagier's rank three examples. We will establish Rogers-Ramanujan type identities involving triple sums for each of them. The sum side of each of these identities is essentially $f_{A,B,C}(q)$ with $(A,B,C)$ from Zagier's examples, and the product side shows clearly that it is a modular form. For example, for Zagier's seventh example we will prove five identities including (see Theorem \ref{thm-7})
\begin{align}
\sum_{i,j,k\geq 0} \frac{q^{i^2+j^2+k^2+ij+ik+\nu(j-k)}}{(q;q)_i(q;q)_j(q;q)_k} &=\frac{(-q^{1+\nu},-q^{1-\nu},q^2;q^2)_\infty}{(q;q)_\infty}, \label{intro-exam7-1} \\
\sum_{i,j,k\geq 0}\frac{q^{i^2+j^2+k^2+ij+ik+i+j}}{(q;q)_i(q;q)_j(q;q)_k}&=\frac{(q^4;q^4)_\infty^2}{(q;q)_\infty (q^2;q^2)_\infty}. \label{intro-exam7-2}
\end{align}
Here \eqref{intro-exam7-1} confirms Zagiers' conjectural identity \eqref{eq-Zagier-conj}. For Zagier's eleventh example we will prove some identities such as (see Theorem \ref{thm-11})
\begin{align}
\sum_{i,j,k\geq 0}\frac{q^{4i^2+2j^2+k^2+4ij-2ik-2jk}}{(q^2;q^2)_i(q^2;q^2)_j(q^2;q^2)_k}&=\frac{(q^2;q^2)_\infty^3(q^5,q^7,q^{12};q^{12})_\infty}{(q;q)_\infty^2 (q^4;q^4)_\infty^2},  \label{intro-exam11-1} \\
\sum_{i,j,k\geq 0} \frac{q^{4i^2+2j^2+k^2+4ij-2ik-2jk+2i}}{(q^2;q^2)_i(q^2;q^2)_j(q^2;q^2)_k}&=\frac{(q^2;q^2)_\infty^3(q^3,q^9,q^{12};q^{12})_\infty}{(q;q)_\infty^2 (q^4;q^4)_\infty^2}, \label{intro-exam11-3}\\
\sum_{i,j,k\geq 0} \frac{q^{4i^2+2j^2+k^2+4ij-2ik-2jk+4i+2j}}{(q^2;q^2)_i(q^2;q^2)_j(q^2;q^2)_k}&=\frac{(q^2;q^2)_\infty^3(q,q^{11},q^{12};q^{12})_\infty}{(q;q)_\infty^2 (q^4;q^4)_\infty^2}. \label{intro-exam11-5}
\end{align}

For most of the examples, the idea of our proof is to reduce triple sum identities to single-sum or double-sum identities.  This can usually be done by summing over some of the summation indices first. However, this does not work for Examples 7 and 10. The reduction processes for these two examples are much more technical. We will cleverly use an integral method and follow techniques in the author's work \cite{Wang}. Surprisingly, for two cases of Example 7, we will prove the desired triple sum identities (see \eqref{intro-exam7-1} and \eqref{intro-exam7-2}) by using a quadruple sum identity. After reducing the sum side to a single sum or double sums, we will either be able to use some known identities or have to establish some new sum-product identities. For instance, for the identity \eqref{intro-exam11-3} we will prove that
\begin{align}
\sum_{n=0}^\infty \frac{q^{n^2}(-q;q^2)_n}{(q^2;q^2)_n} \sum_{i=0}^n q^{2i^2+2i} {n\brack i}_{q^2} &=\frac{(-q;q^2)_\infty (q^3,q^9,q^{12};q^{12})_\infty}{(q^2;q^2)_\infty}. \label{eq-thm-B3}
\end{align}
Here ${n \brack i}_q$ is the $q$-binomial coefficient (see \eqref{q-binomial}).

The rest of this paper is organized as follows. In Section \ref{sec-pre} we collect some useful known identities which are important in proving Zagier's examples. In Section \ref{sec-new} we will establish some new sum-product identities (see Theorems \ref{thm-Wang},  \ref{thm-Bressoud} and \ref{thm-12-lem}). They will be used in Examples 8, 10 and 12 and are of independent interests themselves.  In Section \ref{sec-examples} we will give proofs for Zagier's examples one by one.

\section{Preliminaries}\label{sec-pre}
In this section, we introduce some notations and identities that will be used in our proofs. To make our formulas more compact, sometimes we will use the symbols:
$$J_m:=(q^m;q^m)_\infty, \quad J_{a,m}:=(q^a,q^{m-a},q^m;q^m)_\infty.$$

Recall the $q$-binomial theorem \cite[Theorem 2.1]{Andrews}:
\begin{align}\label{q-binomial}
\sum_{n=0}^\infty \frac{(a;q)_n}{(q;q)_n}z^n=\frac{(az;q)_\infty}{(z;q)_\infty}, \quad |z|<1.
\end{align}
As important corollaries of this theorem, Euler's $q$-exponential identities state that \cite[Corollary 2.2]{Andrews}
\begin{align}\label{Euler}
\sum_{n=0}^\infty \frac{z^n}{(q;q)_n}=\frac{1}{(z;q)_\infty}, \quad \sum_{n=0}^\infty \frac{q^{\binom{n}{2}} z^n}{(q;q)_n}=(-z;q)_\infty, \quad |z|<1.
\end{align}

We define the $q$-binomial coefficient or Gaussian coefficient
$${n\brack m}={n \brack m}_q:=\left\{\begin{array}{ll}
\frac{(q;q)_n}{(q;q)_m (q;q)_{n-m}}, & 0\leq m \leq n, \\
0, & \text{otherwise}. \end{array}\right.$$
One finite version of the second identity in \eqref{Euler} is
\begin{align}\label{eq-finite}
(-z;q)_n=\sum_{i=0}^n { n\brack i}z^i q^{i(i-1)/2}.
\end{align}
The Jacobi triple product identity \cite[Theorem 2.8]{Andrews} is
\begin{align}\label{Jacobi}
(q,z,q/z;q)_\infty=\sum_{n=-\infty}^\infty (-1)^nq^{\binom{n}{2}}z^n.
\end{align}

For Examples 1-3 we need the Durfee rectangle identity: for any fixed integer $n$,
\begin{align}\label{Durfee}
\sum_{j-k=n} \frac{q^{jk}}{(q)_j(q)_k}=\frac{1}{(q;q)_\infty}.
\end{align}

As mentioned in the introduction, some of Zagier's examples can be reduced to some known identities. Besides \eqref{RR-1}--\eqref{Slater16},  we will also need some other identities in the literature:
\begin{align}
&\sum_{n=0}^\infty \frac{q^{n(n+2)}(-q;q^2)_n}{(q^4;q^4)_n} =\frac{J_6J_{12}}{J_{3,12}J_{4,12}}, \quad \text{(Ramanujan \cite[Ent.\  4.2.11]{Lost2}, Stanton \cite{Stanton})} \label{Rama-Stanton} \\
&\sum_{n=0}^\infty \frac{q^{n(n+1)/2}(-1;q)_n}{(q;q)_n(q;q^2)_n}=\frac{J_2J_5^2}{J_1^2J_{10}}, \quad \text{(Rogers \cite[p.\ 330 (4), line 3]{Rogers1917}, corrected)} \label{eq-Rogers-mod10} \\
&\sum_{n=0}^\infty \frac{(-1)^nq^{n^2}(q;q^2)_n}{(-q;q^2)_n(q^4;q^4)_n}=\frac{J_5^2J_{1,10}}{J_{10}J_{2,10}^2}, \quad \text{(S.\ 21)} \label{Slater21} \\
&\sum_{n=0}^\infty \frac{(-1)^nq^{n(n+2)}(q;q^2)_n}{(-q;q^2)_n(q^4;q^4)_n}=\frac{J_1J_5^2J_{2,10}}{J_2^2J_{10}J_{1,5}}, \label{eq-BMS}  \\
& \qquad \qquad \qquad \qquad \qquad  \text{(Bowman-Mc Laughlin-Sills \cite[Eq.\ (2.17)]{BMS})} \nonumber \\
&\sum_{n=0}^\infty \frac{(-1)^nq^{n(n+2)}(q;q^2)_n}{(-q;q^2)_{n+1}(q^4;q^4)_n}=\frac{J_1J_{10}^2}{J_2^2J_5}, \label{eq-MSP}\\
& \qquad \qquad \qquad \qquad \qquad   \text{(Mc Laughlin- Sills-Zimmer \cite[Eq.\ (2.5)]{MSP})}  \nonumber \\
&\sum_{n=0}^\infty \frac{q^{n^2}(-q;q^2)_n}{(q^4;q^4)_n}=\frac{J_6J_{12}J_{3,12}}{J_{1,12}J_{4,12}J_{5,12}}, \quad \text{(Ramanujan \cite[Ent.\  4.2.7]{Lost2}, S.\ 25)} \label{Slater25} \\
&\sum_{n=0}^\infty \frac{q^{n^2+n}(-q^2;q^2)_n}{(q;q)_{2n+1}}=\frac{J_{3,12}}{J_1}, \quad \text{(Ramanujan \cite[Ent.\  4.2.13]{Lost2}, S.\ 28)}, \label{Slater28} \\
&\sum_{n=0}^\infty \frac{q^{2n(n+1)}(q;q^2)_{n+1}}{(q^2;q^2)_{2n+1}}=\sum_{n=0}^\infty \frac{q^{2n(n+1)}}{(q^2;q^2)_n(-q;q)_{2n+1}}=\frac{J_{1,7}}{J_2},  \label{Slater31}  \\
  &\qquad \qquad \qquad \qquad \qquad \qquad \qquad \qquad \text{(Rogers \cite[p.\ 331(6)]{Rogers1894}, S.\ 31)} \nonumber  \\
&\sum_{n=0}^\infty \frac{q^{2n^2+2n}(q;q^2)_n}{(q^2;q^2)_{2n}}=\sum_{n=0}^\infty \frac{q^{2n(n+1)}}{(q^2;q^2)_n(-q;q)_{2n}}=\frac{J_{2,7}}{J_2}, \label{Slater32} \\
  &\qquad \qquad \qquad \qquad \qquad \qquad \qquad \qquad  \text{(Rogers \cite[p.\ 342]{Rogers1917}, S.\ 32)}  \nonumber  \\
&\sum_{n=0}^\infty \frac{q^{2n^2}(q;q^2)_n}{(q^2;q^2)_{2n}}=\sum_{n=0}^\infty \frac{q^{2n^2}}{(q^2;q^2)_n(-q;q)_{2n}}=\frac{J_{3,7}}{J_2},   \label{Slater33} \\
  &\qquad \qquad \qquad \qquad \qquad \qquad \qquad \qquad  \text{(Rogers \cite[p.\ 339]{Rogers1917}, S.\ 33)}  \nonumber \\
&\sum_{n=0}^\infty \frac{q^{n(n+3)/2}(-q;q)_n}{(q;q)_n(q;q^2)_{n+1}}=\frac{J_{10}^3}{J_1J_5J_{3,10}}, \quad \text{(Rogers \cite[p.\ 330 (4), line 2]{Rogers1917}, S.\ 43)} \label{Slater43} \\
&\sum_{n=0}^\infty \frac{q^{n(n+1)/2}(-q;q)_n}{(q;q)_n(q;q^2)_{n+1}}=\frac{J_{10}^3}{J_1J_5J_{1,10}}, \quad \text{(Rogers \cite[p.\ 330 (4), line 1]{Rogers1917}, S.\ 45)} \label{Slater45} \\
&\sum_{n=0}^\infty \frac{q^{n^2+n}(-1;q^2)_n}{(q;q)_{2n}}=\frac{J_6J_{12}^4}{J_{2,12}^2J_{3,12}^2J_{4,12}}. \quad \text{(Ramanujan \cite[Ent.\  4.2.10]{Lost2}, S.\ 48)} \label{Slater48}
\end{align}
Here following the labels in \cite{MSP}, we use the label S.\ $n$ to denote the equation ($n$) in Slater's list \cite{Slater}.
The identities \eqref{Slater31}--\eqref{Slater33}  were later independently rediscovered by Selberg \cite{Selberg} and Dyson \cite{Dyson} and are called as the Rogers-Selberg mod 7 identities.

We will occasionally use some of the above identities with $q$ replaced by $-q$. For instance, for Example 10 we will use not only \eqref{Slater21}--\eqref{eq-MSP}  but also the following identities:
\begin{align}
\sum_{n=0}^\infty \frac{q^{n^2}(-q;q^2)_n}{(q;q^2)_n(q^4;q^4)_n}&=\frac{J_{10}^2J_{20}^5}{J_{1,20}J_{2,20}J_{5,20}^2J_{8,20}^2J_{9,20}}, \label{Slater21-dual} \\
\sum_{n=0}^\infty \frac{q^{n(n+2)}(-q;q^2)_n}{(q;q^2)_n(q^4;q^4)_n}&=\frac{J_{10}^2J_{20}^5}{J_{3,20}J_{4,20}^2J_{5,20}^2J_{6,20}J_{7,20}}, \label{eq-BMS-dual} \\
\sum_{n=0}^\infty \frac{q^{n(n+2)}(-q;q^2)_n}{(q;q^2)_{n+1}(q^4;q^4)_n}&=\frac{J_2J_5J_{20}}{J_1J_4J_{10}}.  \label{eq-MSP-dual}
\end{align}
They are obtained after replacing $q$ by $-q$ in \eqref{Slater21}--\eqref{eq-MSP}, respectively.

For Example 9 we need the Andrews-Gordon identity \cite{Andrews1974,Gordon1961}: for integers $k,s$ such that $k\geq 2$ and $1\leq s \leq k$,
\begin{align}
\sum_{n_1,\cdots,n_{k-1}\geq 0} \frac{q^{N_1^2+\cdots+N_{k-1}^2+N_s+\cdots +N_{k-1}}}{(q;q)_{n_1}(q;q)_{n_2}\cdots (q;q)_{n_{k-1}}}  =\frac{(q^s,q^{2k+1-s},q^{2k+1};q^{2k+1})_\infty}{(q;q)_\infty} \label{AG}
\end{align}
where if $j\leq k-1$, $N_j=n_j+\cdots+n_{k-1}$ and $N_k=0$. This analytical form is due to G.E. Andrews \cite{Andrews1974} and its combinatorial version was earlier given by Gordon \cite{Gordon1961}. This identity is a generalization of the Rogers-Ramanujan identities.

D.M.\ Bressoud \cite{Bressoud1979} gave the even moduli companion of the Andrews-Gordon identity: for positive integer $k\geq 2$ and $1\leq s \leq k$,
\begin{align}\label{eq-Bressoud}
\sum_{n_1,\cdots,n_{k-1}\geq 0} \frac{q^{N_1^2+\cdots+N_{k-1}^2+N_s+\cdots +N_{k-1}}}{(q;q)_{n_1}(q;q)_{n_2}\cdots (q;q)_{n_{k-2}} (q^2;q^2)_{n_{k-1}}} =\frac{(q^s,q^{2k-s},q^{2k};q^{2k})_\infty}{(q;q)_\infty}
\end{align}
where as before, $N_j=n_j+\cdots+n_{k-1}$ if $j\leq k-1$ and $N_k=0$. When $(k,s)=(3,3)$ and $(3,1)$, we get
\begin{align}
\sum_{i,j\geq 0} \frac{q^{i^2+2ij+2j^2}}{(q;q)_i(q^2;q^2)_j}&=\frac{(q^3;q^3)_\infty^2}{(q;q)_\infty (q^6;q^6)_\infty}, \label{Bressoud-cor-1} \\
\sum_{i,j\geq 0} \frac{q^{i^2+2ij+2j^2+i+2j}}{(q;q)_i(q^2;q^2)_j} &=\frac{(q^6;q^6)_\infty^2}{(q^2;q^2)_\infty (q^3;q^3)_\infty}.   \label{Bressoud-cor-2}
\end{align}
See also the author's work \cite[Theorem 1.1]{Wang} for new proofs of these two identities using the integral method described below. These identities will be used in Examples 7 and 8.

For Examples 7 and 10 and some double sums reduced from Example 12 we will use an integral method.  The method goes as follows. First, with the help of \eqref{Euler} and \eqref{Jacobi}, we express the original series as contour integrals of certain infinite products. This step can be done using the following simple fact. For a Laurent series $f(z)=\sum_{n=-N}^\infty a(n)z^n$, if we use $[z^n]f(z)=a(n)$ to denote the coefficient of $z^n$. Then it is well known that
\begin{align}\label{int-constant}
\oint_K f(z) \frac{dz}{2\pi iz}=[z^0]f(z),
\end{align}
where $K$ is a positively oriented and simple closed contour around the origin. The second step is to calculate the integrals. By doing so we will be able to convert the original series to some new series with different sums. We can then evaluate these new series using some known identities.

Example 7 consists of five cases. We need to use double integrals instead of single integral. A key step in using the integral method is the $q$-Gauss summation formula:
\begin{align}\label{q-Gauss}
\sum_{n=0}^\infty \frac{(a,b;q)_n}{(q,c;q)_n}\left(\frac{c}{ab}\right)^n=\frac{(c/a,c/b;q)_\infty}{(c,c/ab;q)_\infty}.
\end{align}
With the help of contour integrals, we are able to reduce three cases of Example 7 to some double sum identities. These include two special cases of Bressoud's identity \eqref{eq-Bressoud} with $(k,s)\in \{(3,1),(3,3)\}$ and a special case of the following identity due to Z.\ Cao and the author \cite[Eq.\ (3.29)]{Cao-Wang}:
\begin{align}\label{eq12-Cao-Wang}
\sum_{i,j\geq 0} \frac{(-1)^i u^{i+2j}q^{i^2+2ij+2j^2-i-j}}{(q;q)_i(q^2;q^2)_j}=(u;q)_\infty.
\end{align}
By the way, for Example 8 (see Theorem \ref{thm-Wang}) we will also need the following companion identity \cite[Eq.\ (3.28)]{Cao-Wang}:
\begin{align}\label{eq-Cao-Wang328}
\sum_{i,j\geq 0} \frac{(-1)^iu^{i+j}q^{i^2+2ij+2j^2}}{(q;q)_i(q^2;q^2)_j}=(uq;q^2)_\infty.
\end{align}

For the remaining two cases of Example 7, surprisingly, we will use integral method to convert them to some quadruple sums. Then we can evaluate these sums utilizing the following identity due to J.\ Dousse and J.\ Lovejoy \cite[Eqs.\ (2.6),(2.7)]{Dousse-Lovejoy}:
\begin{align}\label{DL1112}
    \sum_{i,j,k,\ell\geq 0} \frac{a^{i+\ell}b^{j+\ell}q^{\binom{i+j+k+2\ell+1}{2}+\binom{i+1}{2}+\binom{j+1}{2}+\ell}}{(q;q)_i(q;q)_j(q;q)_k(q^2;q^2)_\ell}=(-q;q)_\infty (-aq^2,-bq^2;q^2)_\infty.
\end{align}
It is worth mentioning that Cao and the author \cite{Cao-Wang} gave a new proof to this identity using the integral method.

For Example 10 we will evaluate the integrals involved using the following result found from the book of G.\ Gasper and M.\ Rahman \cite{GR-book}. Before stating it, we remark that the symbol ``idem $(c_1;c_2,\dots,c_C)$'' after an expression stands for the sum of the $(C-1)$ expressions obtained from the preceding expression by interchanging $c_1$ with each $c_k$, $k=2,3,\dots,C$.

\begin{lemma}\label{lem-integral}
(Cf.\ \cite[Eq.\ (4.10.6)]{GR-book})
Suppose that
$$P(z):=\frac{(a_1z,\dots,a_Az,b_1/z,\dots,b_B/z;q)_\infty}{(c_1z,\dots,c_Cz,d_1/z,\dots,d_D/z;q)_\infty}$$
has only simple poles. We have
\begin{align}\label{eq-integral}
\oint P(z)\frac{dz}{2\pi iz}=& \frac{(b_1c_1,\dots,b_Bc_1,a_1/c_1,\dots,a_A/c_1;q)_\infty }{(q,d_1c_1,\dots,d_Dc_1,c_2/c_1,\dots,c_C/c_1;q)_\infty} \nonumber \\
& \times \sum_{n=0}^\infty \frac{(d_1c_1,\dots,d_Dc_1,qc_1/a_1,\dots,qc_1/a_A;q)_n}{(q,b_1c_1,\dots,b_Bc_1,qc_1/c_2,\dots,qc_1/c_C;q)_n} \nonumber \\
&\times \Big(-c_1q^{(n+1)/2}\Big)^{n(C-A)}\Big(\frac{a_1\cdots a_A}{c_1\cdots c_C} \Big)^n +\text{idem} ~(c_1;c_2,\dots,c_C)
\end{align}
when $C>A$, or if $C=A$ and
\begin{align}\label{cond}
\left|\frac{a_1\cdots a_A }{c_1\cdots c_C}\right|<1.
\end{align}
Here the integration is over a positively oriented contour so that the poles of
$$(c_1z,\dots,c_Cz;q)_\infty^{-1}$$
lie outside the contour, and the origin and poles of $(d_1/z,\dots,d_D/z;q)_\infty^{-1}$ lie inside the contour.
\end{lemma}
This lemma was used in the author's work \cite{Wang} to prove a conjecture of Andrews and Uncu \cite{Andrews-Uncu}.
With the help of the above lemma, we reduce Example 10 to the single-sum identities \eqref{Slater21}--\eqref{eq-MSP} and \eqref{Slater21-dual}--\eqref{eq-MSP-dual}.
\begin{rem}
The integral method was used by Rosengren \cite{Rosengren} in proving some conjectures of S.\ Kanade and M.C.\ Russell \cite{KR-2015,KR-2019}. Later S.\ Chern \cite{Chern} and the author \cite{Wang} used integral method to prove a conjecture of Andrews and A.\ Uncu \cite{Andrews-Uncu}. This method was also applied by J.\ Mc Laughlin  \cite{Laughlin}, and Cao and the author \cite{Cao-Wang} to establish new identities. It also plays an important role in the author's work \cite{Wang2022} in verifying some of Zagier's rank two examples.
\end{rem}

For Example 11 we will need the theory of Bailey pairs. A pair of sequences $\left(\{\alpha_n(x,q)\}_{n=0}^\infty, \{\beta_n(x,q)\}_{n=0}^\infty\right)$ is called a Bailey pair relative to $x$ if
\begin{align}
\beta_n(x,q)=\sum_{r=0}^n \frac{\alpha_r(x,q)}{(q;q)_{n-r}(xq;q)_{n+r}}. \label{Bailey-defn}
\end{align}
Bailey's lemma \cite{Andrews1986} implies many useful transformation formulas, among which we need the following:
\begin{align}
&\sum_{n=0}^\infty x^nq^{n^2}(-q;q^2)_n\beta_n(x,q^2)=\frac{(-xq;q^2)_\infty}{(xq^2;q^2)_\infty} \sum_{r=0}^\infty \frac{x^rq^{r^2}(-q;q^2)_r}{(-xq;q^2)_r}\alpha_r(x,q^2), \label{Bailey-1} \\
&\sum_{n=0}^\infty q^{n(n+1)/2}(-1;q)_n\beta_n(1,q)=2\frac{(q^2;q^2)_\infty}{(q;q)_\infty^2}\sum_{r=0}^\infty \frac{q^{r(r+1)/2}}{1+q^r}\alpha_r(1,q), \label{Bailey-2} \\
&\frac{1}{1-q}\sum_{n=0}^\infty q^{n(n+1)/2}(-q;q)_n\beta_n(q,q)=\frac{(q^2;q^2)_\infty}{(q;q)_\infty^2}\sum_{r=0}^\infty q^{r(r+1)/2}\alpha_r(q,q). \label{Bailey-3}
\end{align}
See \cite[(1.2.9) and (S2BL)]{MSP} for these transformation formulas.

Finally, we will use without details a method of F.\ Garvan and J.\ Liang \cite{Garvan-Liang}. This method is based on the theory of modular forms. It allows us to use some Maple algorithms to automatically verify theta function identities.

\section{New sum-product identities with double sums}\label{sec-new}
We will reduce Examples 8, 11 and 12 to some new identities involving double sums.  In this section, we present and prove these identities.

The following theorem is prepared for Example 8.
\begin{theorem}\label{thm-Wang}
We have
\begin{align}
\sum_{i,j\geq 0} \frac{q^{4i^2+4ij+2j^2}}{(q;q)_{2i}(q^2;q^2)_j} &=\frac{(-q^5,-q^7,q^{12};q^{12})_\infty}{(q^2;q^2)_\infty}, \label{eq-Wang-1} \\
\sum_{i,j\geq 0} \frac{q^{4i^2+4ij+2j^2+4i+2j}}{(q;q)_{2i+1}(q^2;q^2)_j}&=\frac{(-q,-q^{11},q^{12};q^{12})_\infty}{(q^2;q^2)_\infty}. \label{eq-Wang-2}
\end{align}
\end{theorem}
\begin{proof}
Setting $u=1$ in \eqref{eq-Cao-Wang328}, we obtain
\begin{align}
\sum_{i,j\geq 0}\frac{(-1)^iq^{i^2+2ij+2j^2}}{(q;q)_i(q^2;q^2)_j}=\frac{(q;q)_\infty}{(q^2;q^2)_\infty}. \label{Wang-start-2}
\end{align}
Adding \eqref{Bressoud-cor-1} and \eqref{Wang-start-2} together, we deduce that
\begin{align}\label{Wang-proof-1}
\sum_{i,j\geq 0} \frac{q^{4i^2+4ij+2j^2}}{(q;q)_{2i}(q^2;q^2)_j} =\frac{1}{2} \left(\frac{(q^3;q^3)_\infty^2}{(q;q)_\infty (q^6;q^6)_\infty}+ \frac{(q;q)_\infty}{(q^2;q^2)_\infty} \right).
\end{align}
Using the method in \cite{Garvan-Liang}, it is easy to verify that \eqref{eq-Wang-1} holds.

Similarly, subtracting \eqref{Wang-start-2} from \eqref{Bressoud-cor-1}, and using the method in \cite{Garvan-Liang}, it is easy to verify that \eqref{eq-Wang-2} holds.
\end{proof}
\begin{rem}
For the last step, it is not compulsory to use the method in \cite{Garvan-Liang}. We can also finish the proof by $q$-series approaches. For example, after arriving at \eqref{Wang-proof-1}, we can multiply it by $(q^2;q^2)_\infty$. Then the right side will become a nice theta series which is easily shown to be equal to $(-q^5,-q^7,q^{12};q^{12})_\infty$ using \eqref{Jacobi}. However, since such proofs are not always as simple as here, and to save space, we will not discuss them.
\end{rem}

When studying Zagier's Example 11, we encounter Bressoud polynomials:
\begin{align}\label{defn-Bressoud-poly}
B_n^{(1)}(q):=\sum_{k=0}^n q^{k^2}{n \brack k} , \quad B_n^{(2)}(q):=\sum_{k=0}^n q^{k^2+k}{n \brack k} .
\end{align}
We recall the following identities on them, which were implicitly stated as an problem by Andrews \cite{AndrewsSIAM} and later proved by Bressoud \cite{Bressoud}:
\begin{align}
\sum_{k=0}^n q^{k^2} {n \brack k}&=\frac{(q;q)_n}{(q;q)_{2n}}\sum_{k=-\infty}^\infty (-1)^k q^{k(5k+1)/2} {2n \brack n+k}, \label{Bressoud-1} \\
\sum_{k=0}^n q^{k^2+k} {n\brack k}&=\frac{(q;q)_n}{(q;q)_{2n+1}} \sum_{k=-\infty}^\infty (-1)^k q^{k(5k+3)/2} {2n+1 \brack n+k+1}. \label{Bressoud-2}
\end{align}
Along our proofs, we find that \eqref{Bressoud-2} can be written in the following form, which looks closer to \eqref{Bressoud-1}.
\begin{theorem}\label{thm-Bressoud-new}
We have
\begin{align}
\sum_{k=0}^n q^{k^2+k} {n\brack k}&=\frac{(q;q)_n}{(q;q)_{2n}} \sum_{k=-\infty}^\infty (-1)^k q^{k(5k+3)/2}{2n \brack n+k}. \label{new-Bressoud}
\end{align}
\end{theorem}
\begin{proof}[Proof of the equivalence of \eqref{Bressoud-2} and \eqref{new-Bressoud}]
Note that
\begin{align}
{2n+1 \brack n+k+1} =(1-q^{2n+1}) {2n \brack n+k} +q^{n+k+1} {2n+1 \brack n+k+1}. \label{eq-binomial-relation}
\end{align}
Thus the right side of \eqref{Bressoud-2} can be rewritten as
\begin{align}
&\frac{(q;q)_n}{(q;q)_{2n+1}} \sum_{k=-\infty}^\infty (-1)^k q^{k(5k+3)/2} {2n+1 \brack n+k+1}\nonumber \\
&=\frac{(q;q)_n}{(q;q)_{2n}} \sum_{k=-\infty}^\infty (-1)^k q^{k(5k+3)/2} {2n \brack n+k} \nonumber \\
 &\quad +q^{n+1} \frac{(q;q)_n}{(q;q)_{2n+1}} \sum_{k=-\infty}^\infty {2n+1 \brack n+k+1}(-1)^kq^{5k(k+1)/2}. \label{equivalence}
\end{align}
Replacing $k$ by $-k-1$ in the second sum in the right side, we deduce that
\begin{align}
 \sum_{k=-\infty}^\infty {2n+1 \brack n+k+1} (-1)^kq^{5k(k+1)/2}&= \sum_{k=-\infty}^\infty {2n+1 \brack n-k} (-1)^{k+1} q^{5k(k+1)/2} \nonumber \\
 &=- \sum_{k=-\infty}^\infty {2n+1 \brack n+k+1} (-1)^kq^{5k(k+1)/2}.
\end{align}
Hence the second sum in the right side of \eqref{equivalence}  vanishes, and we get the desired equivalence.
\end{proof}
\begin{rem}
S.O.\ Warnaar told us that Theorem \ref{thm-Bressoud-new} was a special case of the following identity (see the third identity on page 681 of the work of Andrews, A.  Schilling, and Warnaar \cite{ASW}):
\begin{align}\label{eq-ASW}
\sum_{r} \frac{(-1)^r q^{((2k-2i+3)(r+1)-2)r/2}}{(q;q)_{n-r}(q;q)_{n+r}} =\sum_{n_1,\dots,n_{k-i}} \frac{q^{n_1^2+\cdots +n_{k-i}^2+n_1+\cdots+n_{k-i}}}{(q;q)_{n-n_1}\cdots (q)_{n_{k-i-1}-n_{k-i}}(q)_{n_{k-i}}}.
\end{align}
If we set $k=2$ and $i=1$ in \eqref{eq-ASW}, then we get Theorem \ref{thm-Bressoud-new}. In a sentence following this identity, they \cite{ASW} also commented that ``and we are back to an equation which implies a Bailey pair relative to $a=1$''. Specifically, this Bailey pair was
\begin{align}
\alpha_n(1;q)&=(-1)^nq^{kn^2+\binom{n}{2}+(k-i+1)n}(1+q^{(2i-2k-1)n}), \label{pair-ASW-alpha} \\
\beta_n(1;q)&=\sum_{n\geq n_1\geq \cdots \geq n_{k-1}\geq 0} \frac{q^{n_1^2+\cdots+n_{k-1}^2+n_i+\cdots +n_{k-1}}}{(q)_{n-n_1}(q)_{n_1-n_2}\cdots (q)_{n_{k-2}-n_{k-1}}(q)_{n_{k-1}}}.
\end{align}
See also Warnaar's survey \cite[p.\ 339]{Warnaar} for this Bailey pair. The Bailey pair \eqref{pair-alpha-2} and \eqref{pair-beta-2} we will use below corresponds to $(k,i)=(2,1)$ of this general one.
\end{rem}

Without assuming \eqref{Bressoud-2}, we can also prove \eqref{new-Bressoud} directly in the following way.
\begin{proof}[Proof of Theorem \ref{thm-Bressoud-new}]
We denote
\begin{align}
L_n(q)&:=\frac{1}{(q;q)_n}\sum_{k=0}^\infty q^{k^2+k} {n\brack k}, \label{Ln-defn} \\
R_n(q)&:=\frac{1}{(q;q)_{2n}}\sum_{k=-\infty}^\infty (-1)^k q^{k(5k+3)/2}{2n \brack n+k}. \label{Rn-defn}
\end{align}
It is easy to find that
\begin{align}\label{initial}
L_0(q)=R_0(q)=1, \quad L_1(q)=R_1(q)=\frac{1+q^2}{1-q}.
\end{align}
Using the Mathematica package $q$-Zeil, we find that
\begin{align}
L_n(q)=\frac{1+q-q^n+q^{2n}}{1-q^n}L_{n-1}(q)-\frac{q}{1-q^n}L_{n-2}(q). \label{L-rec}
\end{align}
Note that
\begin{align}
R_n(q)=\frac{1}{2(q;q)_{2n}}\sum_{k=-\infty}^\infty (-1)^kq^{k(5k-3)/2}(1+q^{3k}){2n \brack n+k}. \label{Rn-new}
\end{align}
Using this expression and the Mathematica package $q$-Zeil, we find that
\begin{align}
R_n(q)=\frac{1+q-q^n+q^{2n}}{1-q^n}R_{n-1}(q)-\frac{q}{1-q^n}R_{n-2}(q). \label{L-rec}
\end{align}
Thus $L_n(q)$ and $R_n(q)$ satisfy the same recurrence relation. In view of the initial values in \eqref{initial}, we see that for any $n\geq 0$, we always have $L_n(q)=R_n(q)$.
\end{proof}
\begin{rem}
If we apply the $q$-Zeil package directly to $R_n(q)$ using its definition \eqref{Rn-defn}, we will get a fifth order recurrence relation. The symmetric version given in \eqref{Rn-new} is motivated by P.\ Paule' proof of \eqref{Bressoud-1} and \eqref{Bressoud-2} in \cite{Paule1994}.
\end{rem}

We will reduce Example 11 to the following identities involving Bressoud polynomials. These identities appear to be new. The identities \eqref{Bressoud-1} and \eqref{Bressoud-2} will be helpful in proving them except \eqref{eq-thm-B3}, for which we need to use \eqref{new-Bressoud} instead.
\begin{theorem}\label{thm-Bressoud}
We have
\begin{align}
\sum_{n=0}^\infty \frac{q^{n(n+1)/2}(-1;q)_n}{(q;q)_n} \sum_{i=0}^n q^{i^2} {n \brack i} &=\frac{(q^2;q^2)_\infty (q^3;q^3)_\infty^2}{(q;q)_\infty^2 (q^6;q^6)_\infty}, \label{eq-thm-B2} \\
\sum_{n=0}^\infty \frac{q^{n(n+1)/2}(-q;q)_n}{(q;q)_n} \sum_{i=0}^n  q^{i^2+i}{n \brack i} &=\frac{(q^6;q^6)_\infty^2}{(q;q)_\infty (q^3;q^3)_\infty}, \label{eq-thm-B4} \\
\sum_{n=0}^\infty \frac{q^{n^2}(-q;q^2)_n}{(q^2;q^2)_n} \sum_{i=0}^n q^{2i^2} {n\brack i}_{q^2} &=\frac{(-q;q^2)_\infty (q^5,q^7,q^{12};q^{12})_\infty}{(q^2;q^2)_\infty}, \label{eq-thm-B1} \\
\sum_{n=0}^\infty \frac{q^{n^2}(-q;q^2)_n}{(q^2;q^2)_n} \sum_{i=0}^n q^{2i^2+2i} {n\brack i}_{q^2} &=\frac{(-q;q^2)_\infty (q^3,q^9,q^{12};q^{12})_\infty}{(q^2;q^2)_\infty}, \label{eq-thm-B3} \\
\sum_{n=0}^\infty \frac{q^{n^2+2n}(-q;q^2)_n}{(q^2;q^2)_n} \sum_{i=0}^n q^{2i^2+2i} {n\brack i}_{q^2} &=\frac{(-q;q^2)_\infty (q^1,q^{11},q^{12};q^{12})_\infty}{(q^2;q^2)_\infty}. \label{eq-thm-B5}
\end{align}
\end{theorem}
\begin{proof}
Note that \eqref{Bressoud-1} implies that $\left(\{\alpha_r(1,q)\}_{r=0}^\infty, \{\beta_r(1,q)\}_{r=0}^\infty\right)$ is a Bailey pair relative to $x=1$, where
\begin{align}
\alpha_r(1,q)&=\left\{\begin{array}{ll}
1, & r=0 \\
(-1)^r q^{r(5r-1)/2}(1+q^r), & r\geq 1,
\end{array}\right.  \label{pair-alpha-1}\\
\beta_r(1,q)&=\frac{1}{(q;q)_r}\sum_{i=0}^r q^{i^2} {r \brack i} . \label{pair-beta-1}
\end{align}
Substituting this Bailey pair into \eqref{Bailey-2}, we deduce that
\begin{align}
\sum_{n=0}^\infty \frac{q^{n(n+1)/2}(-1;q)_n}{(q;q)_n}\sum_{i=0}^n q^{i^2}{n\brack i}&=\frac{(q^2;q^2)_\infty}{(q;q)_\infty^2}\sum_{r=-\infty}^\infty (-1)^r q^{3r^2}.
\end{align}
Now by \eqref{Jacobi} we prove \eqref{eq-thm-B2}.

Substituting this pair with $q$ replaced by $q^2$ into \eqref{Bailey-1}, we deduce that
\begin{align}
\sum_{n=0}^\infty \frac{q^{n^2}(-q;q^2)_n}{(q^2;q^2)_n} \sum_{i=0}^n q^{2i^2}{n\brack i}_{q^2} &=\frac{(-q;q^2)_\infty}{(q^2;q^2)_\infty} \sum_{r=-\infty}^\infty (-1)^r q^{6r^2-r}.
\end{align}
Now by \eqref{Jacobi} we prove \eqref{eq-thm-B1}.

Note that \eqref{Bressoud-2} implies that $\left(\{\alpha_r(q,q)\}_{r=0}^\infty, \{\beta_r(q,q)\}_{r=0}^\infty\right)$ is a Bailey pair relative to $x=q$, where
\begin{align}
\alpha_r(q,q)&=(-1)^r \frac{q^{r(5r+3)/2}(1-q^{2r+1})}{1-q}, \label{pair-alpha-3}\\
\beta_r(q,q)&=\frac{1}{(q;q)_r}\sum_{i=0}^r  q^{i^2+i}{r \brack i} . \label{pair-beta-3}
\end{align}
Substituting this Bailey pair into \eqref{Bailey-3}, we deduce that
\begin{align}
\sum_{n=0}^\infty \frac{q^{n(n+1)/2}(-q;q)_n}{(q;q)_n} \sum_{i=0}^n q^{i^2+i}{n\brack i} &=\frac{(q^2;q^2)_\infty}{(q;q)_\infty^2} \sum_{r=-\infty}^\infty (-1)^r q^{3r^2+2r}.
\end{align}
Now by \eqref{Jacobi} we prove \eqref{eq-thm-B4}.

Substituting this Bailey pair  with $q$ replaced by $q^2$ into \eqref{Bailey-1}, we deduce that
\begin{align}
&\sum_{n=0}^\infty \frac{q^{n^2+2n}(-q;q^2)_n}{(q^2;q^2)_n} \sum_{i=0}^n q^{2i^2+2i} {n\brack i}_{q^2} \nonumber \\ &=\frac{(-q^3;q^2)_\infty}{(q^4;q^2)_\infty} \sum_{r=0}^\infty \frac{q^{r^2+2r}(-q;q^2)_r}{(-q^3;q^2)_r} \cdot \frac{(-1)^r q^{r(5r+3)}(1-q^{4r+2})}{1-q^2} \nonumber \\
&=\frac{(-q;q^2)_\infty}{(q^2;q^2)_\infty} \sum_{r=-\infty}^\infty (-1)^r q^{6r^2+5r}.
\end{align}
Now by \eqref{Jacobi} we prove \eqref{eq-thm-B5}.

Finally, note that Theorem \ref{thm-Bressoud-new} implies that $\left(\{\alpha_r(1,q)\}_{r=0}^\infty, \{\beta_r(1,q)\}_{r=0}^\infty\right)$ is a Bailey pair relative to $x=1$, where
\begin{align}
\alpha_r(1,q)&=\left\{\begin{array}{ll}
1, & r=0 \\
(-1)^r q^{r(5r-3)/2}(1+q^{3r}), & r\geq 1,
\end{array}\right.  \label{pair-alpha-2}\\
\beta_r(1,q)&=\frac{1}{(q;q)_r}\sum_{i=0}^r  q^{i^2+i}{r \brack i}. \label{pair-beta-2}
\end{align}
Substituting this Bailey pair with $q$ replaced by $q^2$ into \eqref{Bailey-1}, we deduce that
\begin{align}
\sum_{n=0}^\infty \frac{q^{n^2}(-q;q^2)_n}{(q^2;q^2)_n} \sum_{i=0}^n q^{2i^2+2i} {n \brack i}_{q^2} &=\frac{(-q;q^2)_\infty}{(q^2;q^2)_\infty} \sum_{r=-\infty}^\infty (-1)^rq^{r(6r-3)}.
\end{align}
Now by \eqref{Jacobi} we prove \eqref{eq-thm-B3}.
\end{proof}
\begin{rem}
Andrews \cite{AndrewsPJM} used a Bailey pair related to the first Bressoud polynomial to derive some multiple Rogers-Ramanujan type identities. The pair used there \cite[(5.6)-(5.7)]{AndrewsPJM} is different from the above.
\end{rem}

For Example 12 we establish the following identities, which give double-sum representations for the products in the Rogers-Ramanujan identities.
\begin{theorem}\label{thm-12-lem}
We have
\begin{align}
\sum_{i,j\geq 0}\frac{q^{\frac{i(3i-1)}{2}+4ij+4j^2}}{(q;q)_i (q^2;q^2)_j} &=\frac{1}{(q,q^4;q^5)_\infty}, \label{eq-lem-12-1} \\
\sum_{i,j\geq 0}\frac{q^{\frac{i(3i+1)}{2}+4ij+4j^2+2j}}{(q;q)_i(q^2;q^2)_j}&=\frac{1}{(q^2,q^3;q^5)_\infty}. \label{eq-lem-12-2}
\end{align}
\end{theorem}
\begin{proof}
We define
\begin{align}
F(u,v;q):=\sum_{i,j\geq 0}\frac{u^iv^jq^{\binom{i}{2}+2\binom{i+2j}{2}}}{(q;q)_i (q^2;q^2)_j}.
\end{align}
By \eqref{Euler} and \eqref{Jacobi} we have
\begin{align}
F(u,v;q)&=\oint \sum_{i=0}^\infty \frac{(-1)^i u^iz^iq^{\binom{i}{2}}}{(q;q)_i} \sum_{j=0}^\infty \frac{v^j z^{2j}}{(q^2;q^2)_j} \sum_{k=-\infty}^\infty (-1)^kq^{2\binom{k}{2}}z^{-k} \frac{\mathrm{d}z}{2\pi iz} \nonumber \\
&=\oint \frac{(uz;q)_\infty (q^2z,1/z,q^2;q^2)_\infty}{(vz^2;q^2)_\infty} \frac{\mathrm{d}z}{2\pi iz}.
\end{align}

Setting $(u,v)=(q,q^2)$, we have
\begin{align}
&F(q,q^2;q)=\oint \frac{(q^2z,1/z,q^2;q^2)_\infty}{(-qz;q)_\infty} \frac{\mathrm{d}z}{2\pi iz} \nonumber \\
&=\oint \sum_{n=0}^\infty \frac{(-1)^nq^nz^n}{(q;q)_n} \sum_{k=-\infty}^\infty (-1)^k q^{k^2-k}z^{-k} \frac{\mathrm{d}z}{2\pi iz} =\sum_{n=0}^\infty \frac{q^{n^2}}{(q;q)_n}.
\end{align}
Now by \eqref{RR-1} we prove \eqref{eq-lem-12-1}.

Setting $(u,v)=(q^2,q^4)$, we have
\begin{align}
&F(q^2,q^4;q)=\oint \frac{(q^2z,1/z,q^2;q^2)_\infty}{(-q^2z;q)_\infty} \frac{\mathrm{d}z}{2\pi iz} \nonumber \\
&=\oint \sum_{n=0}^\infty \frac{(-1)^n q^{2n}z^n}{(q;q)_n} \sum_{k=-\infty}^\infty (-1)^k q^{k^2-k}z^{-k}\frac{\mathrm{d}z}{2\pi iz} =\sum_{n=0}^\infty \frac{q^{n^2+n}}{(q;q)_n}.
\end{align}
Now by \eqref{RR-2} we prove \eqref{eq-lem-12-2}.
\end{proof}

\section{Identities for Zagier's rank three examples}\label{sec-examples}
In this section, we will prove Zagier's rank three examples one by one. Since the scalar part $C$ is easily determined from the matrix part $A$ and the vector part $B$ using the product side of the corresponding identity, we will omit it in most of the time. We will replace the variables $n_1,n_2,n_3$ by $i,j,k$, respectively.

As said in the introduction, Examples 1-3 have been proved by Zagier \cite{Zagier}. For the sake of completeness, we will reproduce his proof and include these examples as well.

\subsection{Example 1.}
The matrix and vector parts for this example are
\begin{align*}
A=\begin{pmatrix} \alpha h^2+1 & \alpha h & -\alpha h \\ \alpha h & \alpha & 1-\alpha \\
-\alpha h & 1-\alpha & \alpha \end{pmatrix}, \quad B \in \left\{\begin{pmatrix}\alpha \nu h \\ \alpha \nu \\ -\alpha \nu \end{pmatrix},
\begin{pmatrix} \alpha \nu h +\frac{1}{2} \\ \alpha \nu  \\ -\alpha \nu \end{pmatrix}, \begin{pmatrix}
\alpha \nu h-1/2 \\ \alpha \nu \\ -\alpha \nu \end{pmatrix}\right\}.
\end{align*}
This example corresponds to $B_1\in \{0,1/2,-1/2\}$ in the following theorem.
\begin{theorem}\label{thm-1}
We have for any $B_1$ that
\begin{align}
&\sum_{i,j,k\geq 0} \frac{q^{\frac{1}{2}(\alpha h^2+1)i^2+\frac{1}{2}\alpha j^2+\frac{1}{2}\alpha k^2+\alpha h ij-\alpha h ik+(1-\alpha)jk+(\alpha \nu h+B_1)i+\alpha \nu j-\alpha \nu k+\nu^2}}{(q;q)_i(q;q)_j(q;q)_k}\nonumber \\
&=\frac{(-q^{\alpha/2},-q^{\alpha/2},q^\alpha;q^\alpha)_\infty (-q^{B_1+\frac{1}{2}};q)_\infty}{(q;q)_\infty}. \label{exam1}
\end{align}
\end{theorem}
\begin{proof}
For any positive number $A_1$, we have
\begin{align}
&\sum_{i,j,k\geq 0} \frac{q^{\frac{1}{2}(\alpha h^2+A_1)i^2+\frac{1}{2}\alpha j^2+\frac{1}{2}\alpha k^2+\alpha h ij-\alpha h ik+(1-\alpha)jk+(\alpha \nu h+B_1)i+\alpha \nu j-\alpha \nu k+\nu^2}}{(q;q)_i(q;q)_j(q;q)_k}\nonumber \\
&=\sum_{i\geq 0} \frac{q^{\frac{A_1}{2}i^2+B_1i}}{(q;q)_i}\sum_{j,k\geq 0} \frac{q^{\frac{\alpha}{2}(hi+j-k+\nu)^2}\cdot q^{jk}}{(q;q)_j(q;q)_k} \nonumber \\
&=\sum_{i\geq 0}  \frac{q^{\frac{A_1}{2}i^2+B_1i}}{(q;q)_i} \sum_{n=-\infty}^\infty q^{\frac{\alpha}{2}(hi+n+\nu)^2} \sum_{j-k=n} \frac{q^{jk}}{(q;q)_j(q;q)_k} \nonumber \\
&=\frac{1}{(q;q)_\infty} \sum_{n=-\infty}^\infty q^{\frac{\alpha}{2}n^2} \sum_{i\geq 0}\frac{q^{\frac{A_1}{2}i^2+B_1i}}{(q;q)_i} \nonumber \\
&=\frac{(-q^{\alpha/2},-q^{\alpha/2},q^\alpha;q^\alpha)_\infty}{(q;q)_\infty} \sum_{i=0}^\infty \frac{q^{\frac{A_1}{2}i^2+B_1i}}{(q;q)_i}. \label{exam1-proof}
\end{align}
Here for the last second equality we used \eqref{Durfee}.

Setting $A_1=1$ in \eqref{exam1-proof}, using \eqref{Euler} we obtain \eqref{exam1}.
\end{proof}
As pointed out by Zagier \cite{Zagier}, from \eqref{exam1-proof} we see that Examples 1-3 correspond to the seven rank one cases (see \eqref{rank1-Rogers} and  \eqref{rank1}).

\subsection{Example 2.}
The matrix and vector parts for this example are
\begin{align*}
A=\begin{pmatrix} \alpha h^2+2 & \alpha h & -\alpha h \\ \alpha h & \alpha & 1-\alpha \\
-\alpha h & 1-\alpha & \alpha \end{pmatrix}, \quad B \in \left\{\begin{pmatrix}\alpha \nu h \\ \alpha \nu \\ -\alpha \nu \end{pmatrix},
\begin{pmatrix} \alpha \nu h +1 \\ \alpha \nu  \\ -\alpha \nu \end{pmatrix}\right\}.
\end{align*}

\begin{theorem}\label{thm-2}
We have for $B_1\in \{0,1\}$,
\begin{align}
&\sum_{i,j,k\geq 0} \frac{q^{\frac{1}{2}(\alpha h^2+2)i^2+\frac{1}{2}\alpha j^2+\frac{1}{2}\alpha k^2+\alpha h ij-\alpha h ik+(1-\alpha)jk+(\alpha \nu h+B_1)i+\alpha \nu j-\alpha \nu k+\nu^2}}{(q;q)_i(q;q)_j(q;q)_k}\nonumber \\
&=\frac{(-q^{\alpha/2},-q^{\alpha/2},q^\alpha;q^\alpha)_\infty }{(q;q)_\infty (q^{1+B_1},q^{4-B_1};q^5)_\infty}. \label{exam2}
\end{align}
\end{theorem}
\begin{proof}
Setting $A_1=2$ and $B_1\in \{0,1\}$ in \eqref{exam1-proof}, using \eqref{RR-1} and \eqref{RR-2} we obtain \eqref{exam2}.
\end{proof}

If we set $h=1$ and $\alpha=2$, we obtain from Theorem \ref{thm-2} that
\begin{align}
\sum_{i,j,k\geq 0} \frac{q^{2i^2+j^2+k^2+2ij-2ik-jk+(2\nu+B_1)i+2\nu j-2\nu k+\nu^2}}{(q;q)_i(q;q)_j(q;q)_k}=\frac{(-q;q^2)_\infty^2 (-q;q)_\infty}{(q^{1+B_1},q^{4-B_1};q^5)_\infty}.
\end{align}
Choosing $B_1=\nu=0$ and interchanging $i$ with $k$, we obtain
\begin{align}
\sum_{i,j,k\geq 0} \frac{q^{i^2+(j+k)^2+k^2-i(j+2k)}}{(q;q)_i(q;q)_j(q;q)_k} =\frac{(-q;q)_\infty (-q;q^2)_\infty^2}{(q,q^4;q^5)_\infty}. \label{eq-Li-Milas}
\end{align}
This identity was first discussed by Andrews \cite[Eq.\ (4.5)]{Andrews1981HJM}. As pointed out by H. Li and A. Milas \cite{Li-Milas}, this identity can be interpreted as elegant product formulas for characters of three different types of level one principal subspaces of certain vertex algebras.

\subsection{Example 3.}
The matrix and vector parts for this example are
\begin{align*}
A=\begin{pmatrix} \alpha h^2+1/2 & \alpha h & -\alpha h \\ \alpha h & \alpha & 1-\alpha \\
-\alpha h & 1-\alpha & \alpha \end{pmatrix}, \quad B \in \left\{\begin{pmatrix}\alpha \nu h \\ \alpha \nu \\ -\alpha \nu \end{pmatrix},
\begin{pmatrix} \alpha \nu h +1/2 \\ \alpha \nu  \\ -\alpha \nu \end{pmatrix}\right\}.
\end{align*}

\begin{theorem}\label{thm-3}
We have
\begin{align}
&\sum_{i,j,k\geq 0} \frac{q^{2(\alpha h^2+\frac{1}{2})i^2+2\alpha j^2+2\alpha k^2+4\alpha h ij-4\alpha h ik+4(1-\alpha)jk+4\alpha \nu hi+4\alpha \nu j-4\alpha \nu k+4\nu^2}}{(q^4;q^4)_i(q^4;q^4)_j(q^4;q^4)_k}\nonumber \\
&=\frac{(-q^{2\alpha},-q^{2\alpha},q^{4\alpha};q^{4\alpha})_\infty J_{10}J_{20}J_{6,20}}{(q^4;q^4)_\infty J_{3,20}J_{7,20}J_{8,20}}, \label{exam3-1} \\
&\sum_{i,j,k\geq 0} \frac{q^{2(\alpha h^2+\frac{1}{2})i^2+2\alpha j^2+2\alpha k^2+4\alpha h ij-4\alpha h ik+4(1-\alpha)jk+4(\alpha \nu h+\frac{1}{2})i+4\alpha \nu j-4\alpha \nu k+4\nu^2}}{(q^4;q^4)_i(q^4;q^4)_j(q^4;q^4)_k}\nonumber \\
&=\frac{(-q^{2\alpha},-q^{2\alpha},q^{4\alpha};q^{4\alpha})_\infty J_{10}J_{20}J_{2,20}}{(q^4;q^4)_\infty J_{1,20}J_{4,20}J_{9,20}}. \label{exam3-2}
\end{align}
\end{theorem}
\begin{proof}
Setting $A_1=\frac{1}{2}$ and $B_1\in \{0,\frac{1}{2}\}$ in \eqref{exam1-proof}, using \eqref{Slater20} and \eqref{Slater16}  we get \eqref{exam3-1} and \eqref{exam3-2}, respectively.
\end{proof}

For the remaining examples, we will repeat the use of symbols like $F(u,v,w;q)$, $R_i(q)$, $S_i(q)$ and $T_i(q)$, etc. They have different meanings in different examples.
\subsection{Example 4.}
The matrix and vector parts for this example are
\begin{align*}
A=\begin{pmatrix} 2 & 1 & -1 \\ 1 & 1 & 0 \\ -1 & 0 & 1 \end{pmatrix}, \quad
B\in \left\{ \begin{pmatrix} 0 \\ -1/2 \\ -1/2 \end{pmatrix}, \begin{pmatrix} 0 \\ 0 \\ 1/2 \end{pmatrix},
\begin{pmatrix} 0 \\ 1/2 \\ 0 \end{pmatrix}, \begin{pmatrix} 1 \\ 1/2 \\ 0 \end{pmatrix},
\begin{pmatrix} 1 \\ 1 \\ -1/2 \end{pmatrix}\right\}.
\end{align*}
It should be noted that $\det A=0$. We see that $A$ is positive semi-definite instead of positive definite. As mentioned in the introduction, we still regard such $(A,B,C)$ as a modular triple.

\begin{theorem}\label{thm-4}
We have
\begin{align}
\sum_{i,j,k\geq 0}\frac{q^{i^2+\frac{1}{2}j^2+\frac{1}{2}k^2+ij-ik-\frac{1}{2}j-\frac{1}{2}k}}{(q;q)_i(q;q)_j(q;q)_k}&=6\frac{J_2^3}{J_1^3}, \label{exam4-1} \\
\sum_{i,j,k\geq 0} \frac{q^{2i^2+j^2+k^2+2ij-2ik+k}}{(q^2;q^2)_i(q^2;q^2)_j(q^2;q^2)_k}&=\frac{J_3^2J_4}{J_1J_2J_6}, \label{exam4-2} \\
\sum_{i,j,k\geq 0}\frac{q^{2i^2+j^2+k^2+2ij-2ik+j}}{(q^2;q^2)_i(q^2;q^2)_j(q^2;q^2)_k}&=\frac{J_2^2J_3^2}{J_1^2J_4J_6}, \label{exam4-3} \\
\sum_{i,j,k\geq 0}\frac{q^{2i^2+j^2+k^2+2ij-2ik+2i+j}}{(q^2;q^2)_i(q^2;q^2)_j(q^2;q^2)_k}&=\frac{J_2J_6^2}{J_1J_3J_4}, \label{exam4-4} \\
\sum_{i,j,k\geq 0} \frac{q^{2i^2+j^2+k^2+2ij-2ik+2i+2j-k}}{(q^2;q^2)_i(q^2;q^2)_j(q^2;q^2)_k}&=2\frac{J_4J_6^2}{J_2^2J_3}. \label{exam4-5}
\end{align}
\end{theorem}
\begin{proof}
We define
\begin{align}
F(u,v,w;q)&:=\sum_{i,j,k\geq 0}\frac{u^iv^jw^kq^{i^2+\frac{1}{2}j^2+\frac{1}{2}k^2+ij-ik}}{(q;q)_i(q;q)_j(q;q)_k}.
\end{align}
We have by \eqref{Euler} that
\begin{align}
F(u,v,w;q)&=\sum_{i\geq 0} \frac{q^{i^2}u^i}{(q;q)_i} \sum_{j\geq 0} \frac{q^{\frac{1}{2}(j^2-j)}\cdot (q^{\frac{1}{2}+i}v)^j}{(q;q)_j} \sum_{k\geq 0} \frac{q^{\frac{1}{2}(k^2-k)}\cdot (q^{\frac{1}{2}-i}w)^k}{(q;q)_k} \nonumber \\
&=\sum_{i\geq 0} \frac{q^{i^2}u^i}{(q;q)_i}(-q^{\frac{1}{2}+i}v;q)_\infty (-q^{\frac{1}{2}-i}w;q)_\infty. \label{exam4-F-start}
\end{align}

The left side of \eqref{exam4-1} corresponds to $F(1,q^{-\frac{1}{2}},q^{-\frac{1}{2}};q)$. By \eqref{exam4-F-start} we have
\begin{align}
F(1,q^{-\frac{1}{2}},q^{-\frac{1}{2}};q)&=2(-q;q)_\infty^2 \sum_{i=0}^\infty \frac{q^{i(i-1)/2}(1+q^i)}{(q;q)_i} \nonumber \\
&=2(-q;q)_\infty^2 \Big((-1;q)_\infty +(-q;q)_\infty \Big)=6(-q;q)_\infty^3.
\end{align}
This proves \eqref{exam4-1}.

The left side of \eqref{exam4-2} corresponds to $F(1,1,q;q^2)$. By \eqref{exam4-F-start} we have
\begin{align}
F(1,1,q;q^2)&=\sum_{i=0}^\infty \frac{q^{2i^2}}{(q^2;q^2)_i}(-q^{1+2i};q^2)_\infty (-q^{2-2i};q^2)_\infty \nonumber \\
&=(-q;q^2)_\infty (-q^2;q^2)_\infty \sum_{i=0}^\infty \frac{q^{i^2+i}(-1;q^2)_i}{(-q;q^2)_i(q^2;q^2)_i}. \label{exam4-2-proof}
\end{align}
Replacing $q$ by $-q$ in \eqref{Slater48} and then substituting it into \eqref{exam4-2-proof}, we obtain \eqref{exam4-2}.

The left side of \eqref{exam4-3} corresponds to $F(1,q,1;q^2)$. By \eqref{exam4-F-start} we have
\begin{align}
F(1,q,1;q^2)&=\sum_{i\geq 0} \frac{q^{2i^2}}{(q^2;q^2)_i}(-q^{2i+2};q^2)_\infty(-q^{1-2i};q^2)_\infty \nonumber \\
&=(-q^2;q^2)_\infty \sum_{i=0}^\infty \frac{q^{2i^2}}{(q^4;q^4)_i}(-q^{1-2i};q^2)_\infty \nonumber \\
&=(-q;q)_\infty \sum_{i=0}^\infty \frac{q^{i^2}(-q;q^2)_i}{(q^4;q^4)_i}.
\end{align}
Now using \eqref{Slater25} we obtain \eqref{exam4-3}.

The left side of \eqref{exam4-4} corresponds to $F(q^2,q,1;q^2)$. By \eqref{exam4-F-start} we have
\begin{align}
F(q^2,q,1;q^2)&=\sum_{i=0}^\infty \frac{q^{2i^2+2i}}{(q^2;q^2)_i}(-q^{2+2i};q^2)_\infty (-q^{1-2i};q^2)_\infty \nonumber \\
&=(-q;q^2)_\infty (-q^2;q^2)_\infty \sum_{i=0}^\infty \frac{q^{i^2+2i}(-q;q^2)_i}{(q^4;q^4)_i}.
\end{align}
Now using \eqref{Rama-Stanton} we obtain \eqref{exam4-4}.

The left side of \eqref{exam4-5} corresponds to $F(q^2,q^2,q^{-1};q^2)$. By \eqref{exam4-F-start} we have
\begin{align}
F(q^2,q^2,q^{-1};q^2)&=\sum_{i=0}^\infty \frac{q^{2i^2+2i}}{(q^2;q^2)_i}(-q^{3+2i};q^2)_\infty (-q^{-2i};q^2)_\infty \nonumber \\
&=(-q;q^2)_\infty (-1;q^2)_\infty \sum_{i=0}^\infty \frac{q^{i^2+i}(-q^2;q^2)_i}{(q^2;q^2)_i (-q;q^2)_{i+1}}.
\end{align}
Now using \eqref{Slater28} with $q$ replaced by $-q$ we obtain \eqref{exam4-5}.
\end{proof}
\subsection{Example 5.}
The matrix and vector parts for this example are
\begin{align*}
A=\begin{pmatrix} 2 & 1 & 1 \\ 1 &1 & 0 \\ 1 & 0 & 1 \end{pmatrix},
\quad B\in \left\{\begin{pmatrix} 0 \\ 0 \\ 1/2 \end{pmatrix}, \begin{pmatrix} 0 \\ 1/2 \\ 0 \end{pmatrix}, \begin{pmatrix} 1 \\ 0 \\ 1/2 \end{pmatrix},
\begin{pmatrix} 1 \\ 1/2 \\ 0 \end{pmatrix}, \begin{pmatrix} 1 \\ 1/2 \\ 1 \end{pmatrix},
\begin{pmatrix} 1 \\ 1 \\ 1/2\end{pmatrix}\right\}.
\end{align*}
Since the quadratic form $n^\mathrm{T}An$ is symmetric in $n_2$ and $n_3$,  there are essentially three identities for this example. Again, we should note that $A$ is positive semi-definite but not positive definite.

\begin{theorem}\label{thm-5}
We have
\begin{align}
\sum_{i,j,k\geq 0}\frac{q^{2i^2+j^2+k^2+2ij+2ik+k}}{(q^2;q^2)_i(q^2;q^2)_j(q^2;q^2)_k}&=\frac{(q^3,q^4,q^7;q^7)_\infty}{(q;q)_\infty}, \label{exam5-1} \\
\sum_{i,j,k\geq 0} \frac{q^{2i^2+j^2+k^2+2ij+2ik+2i+j}}{(q^2;q^2)_i(q^2;q^2)_j(q^2;q^2)_k}&=\frac{(q^2,q^5,q^7;q^7)_\infty}{(q;q)_\infty}, \label{exam5-2} \\
\sum_{i,j,k\geq 0} \frac{q^{2i^2+j^2+k^2+2ij+2ik+2i+j+2k}}{(q^2;q^2)_i(q^2;q^2)_j(q^2;q^2)_k}&=\frac{(q,q^6,q^7;q^7)_\infty}{(q;q)_\infty}. \label{exam5-3}
\end{align}
\end{theorem}
\begin{proof}
We have
\begin{align}
&\sum_{i,j,k\geq 0}\frac{q^{2i^2+j^2+k^2+2ij+2ik+k}}{(q^2;q^2)_i(q^2;q^2)_j(q^2;q^2)_k}=\sum_{i=0}^\infty \frac{q^{2i^2}}{(q^2;q^2)_i} \sum_{j=0}^\infty \frac{q^{j^2+2ij}}{(q^2;q^2)_j} \sum_{k=0}^\infty \frac{q^{k^2+(2i+1)k}}{(q^2;q^2)_k} \nonumber \\
&=\sum_{i=0}^\infty \frac{q^{2i^2}}{(q^2;q^2)_i}(-q^{2i+1};q^2)_\infty (-q^{2i+2};q^2)_\infty
=\sum_{i=0}^\infty \frac{q^{2i^2}}{(q^2;q^2)_i} (-q^{2i+1};q)_\infty \nonumber \\
&=(-q;q)_\infty \sum_{i=0}^\infty \frac{q^{2i^2}}{(q^2;q^2)_i(-q;q)_{2i}}.
\end{align}
This together with \eqref{Slater33} proves \eqref{exam5-1}.

Similarly, we can prove that
\begin{align}
&\sum_{i,j,k\geq 0} \frac{q^{2i^2+j^2+k^2+2ij+2ik+2i+j}}{(q^2;q^2)_i(q^2;q^2)_j(q^2;q^2)_k}=\sum_{i=0}^\infty \frac{q^{2i^2+2i}}{(q^2;q^2)_i}(-q^{2i+1};q)_\infty \nonumber \\
&=(-q;q)_\infty \sum_{i=0}^\infty \frac{q^{2i^2+2i}}{(q^2;q^2)_i(-q;q)_{2i}}.
\end{align}
This together with \eqref{Slater32} proves \eqref{exam5-2}.

In the same way, we have
\begin{align}
&\sum_{i,j,k\geq 0} \frac{q^{2i^2+j^2+k^2+2ij+2ik+2i+j+2k}}{(q^2;q^2)_i(q^2;q^2)_j(q^2;q^2)_k}=\sum_{i=0}^\infty \frac{q^{2i^2+2i}}{(q^2;q^2)_i}(-q^{2i+2};q)_\infty  \nonumber \\
&=(-q;q)_\infty \sum_{i=0}^\infty \frac{q^{2i^2+2i}}{(q^2;q^2)_i(-q;q)_{2i+1}}.
\end{align}
This together with \eqref{Slater31} proves \eqref{exam5-3}.
\end{proof}

\subsection{Example 6.}
The matrix and vector parts for this example are
\begin{align*}
A=\begin{pmatrix} 3 & 2 & 1 \\ 2 & 2 & 1 \\ 1 & 1 & 1 \end{pmatrix}, \quad B\in \left\{ \begin{pmatrix} 0 \\ 0 \\ 0 \end{pmatrix},  \begin{pmatrix} 1/2 \\ 0 \\ 1/2 \end{pmatrix}, \begin{pmatrix}
-1/2 \\ 0 \\ -1/2 \end{pmatrix}, \begin{pmatrix} 1 \\ 1 \\ 0 \end{pmatrix}, \begin{pmatrix} 3/2 \\ 1 \\ 1/2 \end{pmatrix}, \begin{pmatrix} 1/2 \\ 1 \\ -1/2 \end{pmatrix}\right\}.
\end{align*}
Here we rearranged the order of these vectors and corrected a typo for the last vector.

\begin{theorem}\label{thm-6}
We have
\begin{align}
\sum_{i,j,k\geq 0} \frac{q^{\frac{3}{2}i^2+j^2+\frac{1}{2}k^2+2ij+ik+jk}}{(q;q)_i(q;q)_j(q;q)_k}&=\frac{(-q^{\frac{1}{2}};q)_\infty}{(q,q^4;q^5)_\infty}, \label{exam6-1} \\
\sum_{i,j,k\geq 0} \frac{q^{\frac{3}{2}i^2+j^2+\frac{1}{2}k^2+2ij+ik+jk+\frac{1}{2}i+\frac{1}{2}k}}{(q;q)_i(q;q)_j(q;q)_k}&=\frac{(-q;q)_\infty}{(q,q^4;q^5)_\infty}, \label{exam6-2} \\
\sum_{i,j,k\geq 0}\frac{q^{\frac{3}{2}i^2+j^2+\frac{1}{2}k^2+2ij+ik+jk-\frac{1}{2}i-\frac{1}{2}k}}{(q;q)_i(q;q)_j(q;q)_k}&=2\frac{(-q;q)_\infty}{(q,q^4;q^5)_\infty}, \label{exam6-3} \\
\sum_{i,j,k\geq 0}\frac{q^{\frac{3}{2}i^2+j^2+\frac{1}{2}k^2+2ij+ik+jk+i+j}}{(q;q)_i(q;q)_j(q;q)_k}&=\frac{(-q^{\frac{1}{2}};q)_\infty}{(q^2,q^3;q^5)_\infty}, \label{exam6-4} \\
\sum_{i,j,k\geq 0}\frac{q^{\frac{3}{2}i^2+j^2+\frac{1}{2}k^2+2ij+ik+jk+\frac{3}{2}i+j+\frac{1}{2}k}}{(q;q)_i(q;q)_j(q;q)_k}&=\frac{(-q;q)_\infty}{(q^2,q^3;q^5)_\infty}, \label{exam6-5} \\
\sum_{i,j,k\geq 0}\frac{q^{\frac{3}{2}i^2+j^2+\frac{1}{2}k^2+2ij+ik+jk+\frac{1}{2}i+j-\frac{1}{2}k}}{(q;q)_i(q;q)_j(q;q)_k}&=2\frac{(-q;q)_\infty}{(q^2,q^3;q^5)_\infty}. \label{exam6-6}
\end{align}
\end{theorem}
\begin{proof}
We define
\begin{align}
F(u,v,w;q):=\sum_{i,j,k\geq 0}\frac{u^iv^jw^kq^{\frac{3}{2}i^2+j^2+\frac{1}{2}k^2+2ij+ik+jk}}{(q;q)_i(q;q)_j(q;q)_k}.
\end{align}
By \eqref{Euler} we have
\begin{align}
F(u,v,w;q)&=\sum_{i,j\geq 0} \frac{q^{\frac{3}{2}i^2+j^2+2ij}u^iv^j}{(q;q)_i(q;q)_j}\sum_{k\geq 0}\frac{q^{\frac{1}{2}(k^2-k)}(wq^{i+j+\frac{1}{2}})^k}{(q;q)_k} \nonumber \\
&=\sum_{i,j\geq 0}\frac{q^{\frac{3}{2}i^2+j^2+2ij}u^iv^j}{(q;q)_i(q;q)_j}(-wq^{i+j+\frac{1}{2}};q)_\infty \nonumber \\
&=\sum_{n=0}^\infty \sum_{i+j=n} \frac{q^{\frac{1}{2}i^2+(i+j)^2}u^iv^j}{(q;q)_i(q;q)_j}(-wq^{n+\frac{1}{2}};q)_\infty \nonumber \\
&= (-wq^{\frac{1}{2}};q)_\infty \sum_{n=0}^\infty \frac{q^{n^2}}{(-wq^{\frac{1}{2}};q)_n} \sum_{i=0}^n \frac{q^{\frac{1}{2}i^2}u^iv^{n-i}}{(q;q)_i (q;q)_{n-i}} \nonumber \\
&=(-wq^{\frac{1}{2}};q)_\infty \sum_{n=0}^\infty \frac{q^{n^2}v^n}{(q;q)_n(-wq^{\frac{1}{2}};q)_n} \sum_{i=0}^n {n\brack i} \left(\frac{uq^{\frac{1}{2}}}{v}\right)^i q^{i(i-1)/2}\nonumber \\
&=(-wq^{\frac{1}{2}};q)_\infty \sum_{n=0}^\infty \frac{q^{n^2}v^n(-uq^{\frac{1}{2}}/v;q)_n}{(q;q)_n(-wq^{\frac{1}{2}};q)_n}. \quad \text{(by \eqref{eq-finite})}
\end{align}
We consider the case $u=vw$ so that
\begin{align}
F(u,v,w;q)=F(vw,v,w;q)=(-wq^{\frac{1}{2}};q)_\infty\sum_{n=0}^\infty \frac{q^{n^2}v^n}{(q;q)_n}.
\end{align}
When $v=1$ or $q$ by \eqref{RR-1} and \eqref{RR-2} we have
\begin{align}
F(w,1,w;q)&=\frac{(-wq^{\frac{1}{2}};q)_\infty}{(q,q^4;q^5)_\infty}, \\
F(qw,q,w;q)&=\frac{(-wq^{\frac{1}{2}};q)_\infty}{(q^2,q^3;q^5)_\infty}.
\end{align}
Now with $w=1, q^{\frac{1}{2}}, q^{-\frac{1}{2}}$, we get the identities \eqref{exam6-1}--\eqref{exam6-6}.
\end{proof}

\subsection{Example 7.}
The matrix and vector parts for this example are
\begin{align*}
A=\begin{pmatrix} 2 &1 &1 \\ 1 & 2 & 0 \\ 1 & 0 & 2 \end{pmatrix}, \quad
& B\in \Bigg\{\begin{pmatrix} 0 \\ \nu \\ -\nu \end{pmatrix}, \begin{pmatrix} 1/2 \\ 0 \\ 1 \end{pmatrix},
\begin{pmatrix} 1/2 \\ 1 \\ 0  \end{pmatrix}, \begin{pmatrix} 1 \\ 0 \\ 1 \end{pmatrix}, \begin{pmatrix} 1\\ 1 \\ 0  \end{pmatrix}, \nonumber \\
&\quad \begin{pmatrix} -1/2 \\ 0 \\0   \end{pmatrix}, \begin{pmatrix} -1 \\ -1/2 \\ -1/2 \end{pmatrix}\Bigg\}.
\end{align*}
Since the quadratic form $n^\mathrm{T}An$ is symmetric in $n_2$ and $n_3$, there are essentially five cases for this example. Zagier stated the conjectural identity \eqref{eq-Zagier-conj} for the first value of $B$, which is equivalent to \eqref{exam7-1} below.
\begin{theorem}\label{thm-7}
We have
\begin{align}
\sum_{i,j,k\geq 0} \frac{q^{i^2+j^2+k^2+ij+ik+\nu(j-k)}}{(q;q)_i(q;q)_j(q;q)_k} &=\frac{(-q^{1+\nu},-q^{1-\nu},q^2;q^2)_\infty}{(q;q)_\infty}, \label{exam7-1} \\
\sum_{i,j,k\geq 0}\frac{q^{i^2+j^2+k^2+ij+ik+i+j}}{(q;q)_i(q;q)_j(q;q)_k}&=\frac{(q^4;q^4)_\infty^2}{(q;q)_\infty (q^2;q^2)_\infty}, \label{exam7-2} \\
\sum_{i,j,k\geq 0}\frac{q^{2i^2+2j^2+2k^2+2ij+2ik+i+2k}}{(q^2;q^2)_i(q^2;q^2)_j(q^2;q^2)_k}&=\frac{(q,q^5,q^6;q^6)_\infty}{(q;q)_\infty}, \label{exam7-3} \\
\sum_{i,j,k\geq 0}\frac{q^{2i^2+2j^2+2k^2+2ij+2ik-i}}{(q^2;q^2)_i(q^2;q^2)_j(q^2;q^2)_k}&=\frac{(q^3;q^3)_\infty^2}{(q;q)_\infty (q^6;q^6)_\infty},  \label{exam7-4}\\
\sum_{i,j,k\geq 0}\frac{q^{2i^2+2j^2+2k^2+2ij+2ik-2i-j-k}}{(q^2;q^2)_i(q^2;q^2)_j(q^2;q^2)_k}&=2\frac{(q^2;q^2)_\infty}{(q;q)_\infty}. \label{exam7-5}
\end{align}
\end{theorem}
\begin{proof}
We define
\begin{align}
F(u,v,w;q)&:=\sum_{i,j,k\geq 0}\frac{q^{2i^2+2j^2+2k^2+2ij+2ik}u^iv^jw^k}{(q^2;q^2)_i(q^2;q^2)_j(q^2;q^2)_k}.
\end{align}
By \eqref{Euler} and \eqref{Jacobi} we have
\begin{align}
&F(u,v,w;q)=\sum_{i,j,k\geq 0}\frac{q^{(i+j+k)^2+i^2+(j-k)^2}u^iv^jw^k}{(q^2;q^2)_i(q^2;q^2)_j(q^2;q^2)_k} \nonumber \\
&=\oint\oint \sum_{i\geq 0}\frac{q^{i^2}z^iu^i}{(q^2;q^2)_i} \sum_{j\geq 0} \frac{z^jt^jv^j}{(q^2;q^2)_j} \sum_{k\geq 0} \frac{z^kw^kt^{-k}}{(q^2;q^2)_k} \sum_{\ell=-\infty} z^{-\ell} q^{\ell^2} \sum_{m=-\infty}^\infty t^{-m}q^{m^2} \frac{\mathrm{d}t}{2\pi it} \frac{\mathrm{d}z}{2\pi iz} \nonumber \\
&=\oint\oint \frac{(-quz,-qz,-q/z,q^2;q^2)_\infty (-qt,-q/t,q^2;q^2)_\infty}{(vzt,wz/t;q^2)_\infty} \frac{\mathrm{d}t}{2\pi it} \frac{\mathrm{d}z}{2\pi iz} \nonumber \\
&=\oint (-quz,-qz,-q/z,q^2,q^2;q^2)_\infty \oint \sum_{i=0}^\infty \frac{(-\frac{q}{vz};q^2)_i(vz)^it^i}{(q^2;q^2)_i} \nonumber \\
&\quad \quad \quad \quad \cdot \sum_{j=0}^\infty \frac{(-\frac{q}{wz};q^2)_j(wz)^jt^{-j}}{(q^2;q^2)_j} \frac{\mathrm{d}t}{2\pi it} \frac{\mathrm{d}z}{2\pi iz} \quad \text{(by \eqref{q-binomial})} \nonumber \\
&=\oint (-quz,-qz,-q/z,q^2,q^2;q^2)_\infty \sum_{i=0}^\infty \frac{(-\frac{q}{vz},-\frac{q}{wz};q^2)_i}{(q^2,q^2;q^2)_i} z^{2i}v^iw^i \frac{\mathrm{d}z}{2\pi iz} \quad \text{(by \eqref{q-Gauss})} \nonumber \\
&=\oint (-quz,-qz,-q/z,q^2,q^2;q^2)_\infty \cdot \frac{(-qvz,-qwz;q^2)_\infty}{(q^2,z^2vw;q^2)_\infty} \frac{\mathrm{d}z}{2\pi iz} \nonumber \\
&=\oint \frac{(-quz,-qvz,-qwz,-qz,-q/z,q^2;q^2)_\infty}{(z^2vw;q^2)_\infty} \frac{\mathrm{d}z}{2\pi iz}. \label{exam7-F-start}
\end{align}
Now we treat the five identities separately.

(1) The left side of \eqref{exam7-1} is just $F(1,q^{\nu},q^{-\nu};q^{\frac{1}{2}})$. By \eqref{exam7-F-start} we have
\begin{align}
&F(1,q^{\nu},q^{-\nu};q^{\frac{1}{2}})=\oint \frac{(-q^{\frac{1}{2}}z,-q^{\frac{1}{2}+\nu}z,-q^{\frac{1}{2}-\nu}z,-q^{\frac{1}{2}}z,-q^{\frac{1}{2}}/z,q;q)_\infty}{(z^2;q^2)_\infty (qz^2;q^2)_\infty} \frac{\mathrm{d}z}{2\pi iz} \nonumber \\
&=\oint \frac{(-q^{\frac{1}{2}+\nu}z,-q^{\frac{1}{2}-\nu}z,-q^{\frac{1}{2}}z,-q^{\frac{1}{2}}/z,q;q)_\infty}{(q^{\frac{1}{2}}z;q)_\infty (z^2;q^2)_\infty}\frac{\mathrm{d}z}{2\pi iz} \nonumber \\
&=\oint \sum_{i\geq 0} \frac{q^{\frac{i^2}{2}+\nu i}z^i}{(q;q)_i} \sum_{j\geq 0} \frac{q^{\frac{j^2}{2}-\nu j}z^j}{(q;q)_j} \sum_{k\geq 0} \frac{q^{\frac{k}{2}}z^k}{(q;q)_k} \sum_{\ell\geq 0} \frac{z^{2\ell}}{(q^2;q^2)_\ell} \sum_{m=-\infty}^\infty z^{-m} q^{\frac{m^2}{2}} \frac{\mathrm{d}z}{2\pi iz} \nonumber \\
&\qquad \qquad \qquad \qquad \qquad \qquad \qquad \qquad  \text{(by \eqref{Euler} and \eqref{Jacobi})} \nonumber \\
&=\sum_{i,j,k,\ell \geq 0} \frac{q^{\frac{1}{2}i^2+\frac{1}{2}j^2+\frac{1}{2}(i+j+k+2\ell)^2+\nu(i-j)+\frac{1}{2}k}}{(q;q)_i(q;q)_j(q;q)_k(q^2;q^2)_\ell} \nonumber \\
&=\sum_{i,j,k,\ell \geq 0} \frac{q^{(\nu-1)(i+\ell)} \cdot q^{(-\nu-1)(j+\ell)} \cdot q^{\binom{i+j+k+2\ell+1}{2}+\binom{i+1}{2}+\binom{j+1}{2}+\ell}}{(q;q)_i(q;q)_j(q;q)_k(q^2;q^2)_\ell} \nonumber \\
&=(-q;q)_\infty (-q^{1+\nu},-q^{1-\nu};q^2)_\infty. \quad \text{(by \eqref{DL1112} with $a=q^{\nu-1}$, $b=q^{-\nu-1}$)}
\end{align}
This proves \eqref{exam7-1}.

(2) The left side of \eqref{exam7-2} is just $F(q,q,1;q^{\frac{1}{2}})$. By \eqref{exam7-F-start} we have
\begin{align}
&F(q,q,1;q^{\frac{1}{2}})=\oint \frac{(-q^{\frac{3}{2}}z,-q^{\frac{3}{2}}z,-q^{\frac{1}{2}}z,-q^{\frac{1}{2}}z,-q^{\frac{1}{2}}/z,q;q)_\infty}{(qz^2,q^2z^2;q^2)_\infty} \frac{\mathrm{d}z}{2\pi iz} \nonumber \\
&=\oint \frac{(-q^{\frac{3}{2}}z,-q^{\frac{3}{2}}z,-q^{\frac{1}{2}}z,-q^{\frac{1}{2}}/z,q;q)_\infty}{(q^{\frac{1}{2}}z;q)_\infty (q^2z^2;q^2)_\infty} \frac{\mathrm{d}z}{2\pi iz} \nonumber \\
&=\oint \sum_{i\geq 0} \frac{q^{\frac{1}{2}i^2+i}z^i}{(q;q)_i} \sum_{j\geq 0} \frac{q^{\frac{1}{2}j^2+j}z^j}{(q;q)_j} \sum_{k\geq 0}\frac{q^{\frac{k}{2}}z^k}{(q;q)_k} \sum_{\ell \geq 0} \frac{q^{2\ell}z^{2\ell}}{(q^2;q^2)_{\ell}} \sum_{m=-\infty}^\infty z^{-m}q^{\frac{m^2}{2}} \frac{\mathrm{d}z}{2\pi iz} \nonumber \\
&\quad \quad \quad \quad \quad \quad \text{(by \eqref{Euler} and \eqref{Jacobi})}  \nonumber \\
&=\sum_{i,j,k,\ell\geq 0} \frac{q^{\frac{1}{2}i^2+\frac{1}{2}j^2+\frac{1}{2}(i+j+k+2\ell)^2+i+j+\frac{1}{2}k+2\ell}}{(q;q)_i(q;q)_j(q;q)_k(q^2;q^2)_\ell} \nonumber \\
&=\sum_{i,j,k,\ell \geq 0} \frac{q^{\binom{i+j+k+2\ell+1}{2}+\binom{i+1}{2}+\binom{j+1}{2}+\ell}}{(q;q)_i(q;q)_j(q;q)_k(q^2;q^2)_\ell} \nonumber \\
&=(-q;q)_\infty (-q^2,-q^2;q^2)_\infty. \quad \text{(by \eqref{DL1112} with $a=b=1$)}
\end{align}
This proves \eqref{exam7-2}.

(3) The left side of \eqref{exam7-3} is just $F(q,1,q^2;q)$. By \eqref{exam7-F-start} we have
\begin{align}
&F(q,1,q^2;q)=\oint \frac{(-q^2z,-qz,-q^3z,-qz,-q/z,q^2;q^2)_\infty}{(q^2z^2;q^2)_\infty} \frac{\mathrm{d}z}{2\pi iz} \nonumber \\
&=\oint \frac{(-qz;q)_\infty (-q^3z,-qz,-q/z,q^2;q^2)_\infty}{(qz,-qz;q)_\infty} \frac{\mathrm{d}z}{2\pi iz} \nonumber \\
&=\oint \frac{(-q^3z,-qz,-q/z,q^2;q^2)_\infty}{(qz;q)_\infty} \frac{\mathrm{d}z}{2\pi iz} \nonumber \\
&=\oint \sum_{i\geq 0} \frac{z^iq^i}{(q;q)_i} \sum_{j\geq 0} \frac{z^jq^{j^2+2j}}{(q^2;q^2)_j} \sum_{k=-\infty}^\infty z^{-k}q^{k^2} \frac{\mathrm{d}z}{2\pi iz} \quad \text{(by \eqref{Euler} and \eqref{Jacobi})} \nonumber \\
&=\sum_{i,j\geq 0} \frac{q^{i^2+2ij+2j^2+i+2j}}{(q;q)_i(q^2;q^2)_j} =\frac{(q^6;q^6)_\infty^2}{(q^2;q^2)_\infty (q^3;q^3)_\infty}.  \quad \text{(by \eqref{Bressoud-cor-2})} \label{exam7-2-last}
\end{align}
This proves \eqref{exam7-3}.

(4) The left side of \eqref{exam7-4} is just $F(q^{-1},1,1;q)$. By \eqref{exam7-F-start} we have
\begin{align}
&F(q^{-1},1,1;q)=\oint \frac{(-z,-qz,-qz,-qz,-q/z,q^2;q^2)_\infty}{(z^2;q^2)_\infty} \frac{\mathrm{d}z}{2\pi iz} \nonumber \\
&=\oint \frac{(-qz,-qz,-q/z,q^2;q^2)_\infty}{(z;q)_\infty} \frac{\mathrm{d}z}{2\pi iz} \nonumber \\
&=\oint \sum_{i\geq 0} \frac{z^i}{(q;q)_i} \sum_{j\geq 0} \frac{z^j q^{j^2}}{(q^2;q^2)_j} \sum_{k=-\infty}^\infty z^{-k}q^{k^2} \frac{\mathrm{d}z}{2\pi iz} \quad \text{(by \eqref{Euler} and \eqref{Jacobi})} \nonumber \\
&=\sum_{i,j\geq 0} \frac{q^{i^2+2ij+2j^2}}{(q;q)_i(q^2;q^2)_j} =\frac{(q^3;q^3)_\infty^2}{(q;q)_\infty (q^6;q^6)_\infty}. \quad \text{(by \eqref{Bressoud-cor-1})} \label{exam7-4-last}
\end{align}
This proves \eqref{exam7-4}.

(5) The left side of \eqref{exam7-5} is just $F(q^{-2},q^{-1},q^{-1};q)$. By \eqref{exam7-F-start} we have
\begin{align}
&F(q^{-2},q^{-1},q^{-1};q)=\oint \frac{(-q^{-1}z,-z,-z,-qz,-q/z,q^2;q^2)_\infty}{(q^{-2}z^2;q^2)_\infty} \frac{\mathrm{d}z}{2\pi iz} \nonumber \\
&=\oint \frac{(-z,-qz,-q/z,q^2;q^2)_\infty}{(q^{-1}z;q)_\infty} \frac{\mathrm{d}z}{2\pi iz} \nonumber \\
&=\oint \sum_{i\geq 0} \frac{q^{-i}z^i}{(q;q)_i} \sum_{j\geq 0} \frac{q^{j^2-j}z^j}{(q^2;q^2)_j} \sum_{k=-\infty}^\infty z^{-k}q^{k^2} \frac{\mathrm{d}z}{2\pi iz}  \quad  \text{(by \eqref{Euler} and \eqref{Jacobi})} \nonumber \\
&=\sum_{i,j\geq 0} \frac{q^{i^2+2ij+2j^2-i-j}}{(q;q)_i(q^2;q^2)_j} =(-1;q)_\infty. \quad \text{(by \eqref{eq12-Cao-Wang} with $u=-1$)}
\end{align}
This proves \eqref{exam7-5}.
\end{proof}

Since the value of $\nu$ in \eqref{exam7-1} is arbitrary, we actually proved the following result.
\begin{corollary}\label{cor-7}
We have
\begin{align}
\sum_{i,j,k\geq 0} \frac{q^{i^2+j^2+k^2+i(j+k)}}{(q;q)_i(q;q)_j(q;q)_k} a^{j-k}=\frac{(-aq,-q/a,q^2;q^2)_\infty}{(q;q)_\infty}.
\end{align}
\end{corollary}
By \eqref{Jacobi}, we have
\begin{align}\label{cor-7-R-coeff}
[a^r] \frac{(-aq,-q/a,q^2;q^2)_\infty}{(q;q)_\infty}=\frac{q^{r^2}}{(q;q)_\infty}.
\end{align}
On the other hand, we have
\begin{align}\label{cor-7-L-coeff}
[a^r]\sum_{i,j,k\geq 0} \frac{q^{i^2+j^2+k^2+i(j+k)}}{(q;q)_i(q;q)_j(q;q)_k} a^{j-k} =\sum_{i=0}^\infty \frac{q^{i^2}}{(q;q)_i} \sum_{j-k=r} \frac{q^{j^2+k^2+i(j+k)}}{(q;q)_j(q;q)_k}.
\end{align}
By the truth of \eqref{exam7-1}, we know that \eqref{cor-7-R-coeff} is the same with \eqref{cor-7-L-coeff}. It would be quite interesting if one can prove this fact directly, which will yield a new proof for \eqref{exam7-1}.

\subsection{Example 8.}
The matrix and vector parts for this example are
$$A=\begin{pmatrix}
4 & 2 & 1 \\ 2 & 2 & 0 \\ 1 & 0 & 1
\end{pmatrix}, \quad B\in \left\{\begin{pmatrix} -1 \\ -1 \\ 1/2 \end{pmatrix}, \begin{pmatrix} 0 \\ -1/2 \\ 1/2 \end{pmatrix}, \begin{pmatrix} 0 \\ 0 \\ 1/2 \end{pmatrix},
\begin{pmatrix} 2 \\ 1 \\ 1/2 \end{pmatrix}, \begin{pmatrix} 0 \\ 0 \\ 0 \end{pmatrix},  \begin{pmatrix} 2 \\ 1 \\ 1 \end{pmatrix}\right\}.$$
We find a missing case. Namely, $f_{A,B,C}(q)$ is also modular for $B=(1,0,1/2)^\mathrm{T}$ and $C=1/12$.
\begin{theorem}\label{thm-8}
We have
\begin{align}
\sum_{i,j,k\geq 0} \frac{q^{2i^2+j^2+\frac{1}{2}k^2+2ij+ik-i-j+\frac{1}{2}k}}{(q;q)_i(q;q)_j(q;q)_k}&=2\frac{(q^2;q^2)_\infty^2}{(q;q)_\infty^2}, \label{exam8-1} \\
\sum_{i,j,k\geq 0}\frac{q^{2i^2+j^2+\frac{1}{2}k^2+2ij+ik-\frac{1}{2}j+\frac{1}{2}k}}{(q;q)_i(q;q)_j(q;q)_k}&=\frac{1}{(q^{\frac{1}{2}};q)_\infty}, \label{exam8-2} \\
\sum_{i,j,k\geq 0} \frac{q^{2i^2+j^2+\frac{1}{2}k^2+2ij+ik+i+\frac{1}{2}k}}{(q;q)_i(q;q)_j(q;q)_k} &=\frac{(q^2;q^2)_\infty^2}{(q;q)_\infty^2}, \label{exam8-new} \\
\sum_{i,j,k\geq 0}\frac{q^{2i^2+j^2+\frac{1}{2}k^2+2ij+ik+\frac{1}{2}k}}{(q;q)_i(q;q)_j(q;q)_k}&=\frac{(q^2;q^2)_\infty (q^3;q^3)_\infty^2}{(q;q)_\infty^2 (q^6;q^6)_\infty}, \label{exam8-3} \\
\sum_{i,j,k\geq 0}\frac{q^{2i^2+j^2+\frac{1}{2}k^2+2ij+ik+2i+j+\frac{1}{2}k}}{(q;q)_i(q;q)_j(q;q)_k}&=\frac{(q^6;q^6)_\infty^2}{(q;q)_\infty (q^3;q^3)_\infty}, \label{exam8-4} \\
\sum_{i,j,k\geq 0}\frac{q^{4i^2+2j^2+k^2+4ij+2ik}}{(q^2;q^2)_i(q^2;q^2)_j(q^2;q^2)_k}&=\frac{1}{(q,q^3,q^4,q^8,q^9,q^{11};q^{12})_\infty}, \label{exam8-5}   \\
\sum_{i,j,k\geq 0}\frac{q^{4i^2+2j^2+k^2+4ij+2ik+4i+2j+2k}}{(q^2;q^2)_i(q^2;q^2)_j(q^2;q^2)_k}&=\frac{1}{(q^3,q^4,q^5,q^7,q^8,q^9;q^{12})_\infty}. \label{exam8-6}
\end{align}
\end{theorem}
\begin{proof}
We consider
\begin{align}
F(u,v,w;q):=\sum_{i,j,k\geq 0}\frac{q^{2i^2+j^2+\frac{1}{2}k^2+2ij+ik}u^iv^jw^k}{(q;q)_i(q;q)_j(q;q)_k}.
\end{align}
By \eqref{Euler} we have
\begin{align}
F(u,v,w;q)&=\sum_{i,j\geq 0}\frac{u^iv^jq^{2i^2+j^2+2ij}}{(q;q)_i(q;q)_j} \sum_{k\geq 0} \frac{q^{\frac{1}{2}k^2-\frac{1}{2}k}\cdot (wq^{i+\frac{1}{2}})^k}{(q;q)_k} \nonumber \\
&=\sum_{i,j\geq 0}\frac{u^iv^jq^{2i^2+2ij+j^2}}{(q;q)_i (q;q)_j} (-wq^{\frac{1}{2}+i};q)_\infty. \label{exam8-F}
\end{align}
Setting $(u,v,w)$ as $(t^2q^{-1},tq^{-1},q^{\frac{1}{2}})$, by \eqref{exam8-F} we obtain
\begin{align}
&F(t^2q^{-1},tq^{-1},q^{\frac{1}{2}};q)=\sum_{i,j\geq 0}\frac{t^{2i+j}q^{2i^2+2ij+j^2-i-j}}{(q;q)_i (q;q)_j} (-q^{1+i};q)_\infty \nonumber \\
&=(-q;q)_\infty \sum_{i,j\geq 0}\frac{t^{2i+j}q^{2i^2+2ij+j^2-i-j}}{(q^2;q^2)_i(q;q)_j} =(-q;q)_\infty (-t;q)_\infty. \label{exam8-start}
\end{align}
Here for the last equality we used \eqref{eq12-Cao-Wang} with $i,j$ interchanged.

Setting $t=1,q^{\frac{1}{2}}$ and $q$ in \eqref{exam8-start}, we obtain \eqref{exam8-1}, \eqref{exam8-2} and \eqref{exam8-new}, respectively.

Setting $(u,v,w)=(1,1,q^{\frac{1}{2}})$, by \eqref{exam8-F} and \eqref{Bressoud-cor-1} we have
\begin{align}
F(1,1,q^{\frac{1}{2}};q)&=(-q;q)_\infty \sum_{i,j\geq 0} \frac{q^{2i^2+2ij+j^2}}{(q^2;q^2)_i(q;q)_j}=(-q;q)_\infty \frac{(q^3;q^3)_\infty^2}{(q;q)_\infty (q^6;q^6)_\infty}.
\end{align}
This proves \eqref{exam8-3}.

Setting $(u,v,w)=(q^2,q,q^{\frac{1}{2}})$, by \eqref{exam8-F} and \eqref{Bressoud-cor-2} we have
\begin{align}
F(q^2,q,q^{\frac{1}{2}};q)=(-q;q)_\infty \sum_{i,j\geq 0} \frac{q^{2i^2+2ij+j^2+2i+j}}{(q^2;q^2)_i(q;q)_j}=(-q;q)_\infty \frac{(q^6;q^6)_\infty^2}{(q^2;q^2)_\infty (q^3;q^3)_\infty}.
\end{align}
This proves \eqref{exam8-4}.

By \eqref{exam8-F} we have
\begin{align}\label{exam8-5-proof}
F(1,1,1;q^2)&=\sum_{i,j\geq 0} \frac{q^{4i^2+4ij+2j^2}}{(q^2;q^2)_i (q^2;q^2)_j} \cdot \frac{(-q;q^2)_\infty}{(-q;q^2)_i}=:T(q).
\end{align}
Replacing $q$ by $-q$, we have
\begin{align}
T(-q)&=(q;q^2)_\infty \sum_{i,j\geq 0}\frac{q^{4i^2+4ij+2j^2}}{(q;q)_{2i} (q^2;q^2)_{j}}=\frac{(q;q^2)_\infty (-q^5,-q^7,q^{12};q^{12})_\infty}{(q^2;q^2)_\infty}.
\end{align}
Here we used \eqref{eq-Wang-1}. Now replacing $q$ back to $-q$ and using \eqref{exam8-5-proof}, we get \eqref{exam8-5}.

Next, by \eqref{exam8-F} we have
\begin{align}\label{exam8-6-proof}
F(q^4,q^2,q^2;q^2)&=(-q;q^2)_\infty \sum_{i,j\geq 0} \frac{q^{4i^2+4ij+2j^2+4i+2j}}{(q^2;q^2)_i(q^2;q^2)_j(-q;q^2)_{i+1}}=:R(q).
\end{align}
Replacing $q$ by $-q$, we get by \eqref{eq-Wang-2} that
\begin{align}
R(-q)=(q;q^2)_\infty \sum_{i,j\geq 0} \frac{q^{4i^2+4ij+2j^2+4i+2j}}{(q;q)_{2i+1} (q^2;q^2)_{j}}=\frac{(q;q^2)_\infty (-q,-q^{11},q^{12};q^{12})_\infty}{(q^2;q^2)_\infty}.
\end{align}
Replacing $q$ back to $-q$ and using \eqref{exam8-6-proof}, we obtain \eqref{exam8-6}.
\end{proof}

\begin{rem}
From \eqref{exam8-F} we deduce that
\begin{align}\label{exam8-rem}
F(u,-u,q^{\frac{1}{2}};q)=(-q;q)_\infty \sum_{i,j,k\geq 0} \frac{(-1)^j u^{i+j}q^{2i^2+2ij+j^2}}{(q^2;q^2)_i (q;q)_j} =(-q;q)_\infty (uq;q^2)_\infty.
\end{align}
Here the last equality follows from \eqref{eq-Cao-Wang328}. This will produce some modular forms for $u=\pm 1, \pm q, \pm q^{-1}$. However, the left side is an alternating series and not of the form $f_{A,B,C}(q)$. Nevertheless, we get the following curious identity when $u=1$:
\begin{align}\label{eq-exam8-curious}
\sum_{i,j,k\geq 0}\frac{(-1)^jq^{2i^2+j^2+\frac{1}{2}k^2+2ij+ik+\frac{1}{2}k}}{(q;q)_i(q;q)_j(q;q)_k}=1.
\end{align}
\end{rem}

\subsection{Example 9.}
The matrix and vector parts for this example are
\begin{align*}
A=\begin{pmatrix} 6 &4 & 2 \\ 4 & 4 & 2 \\ 2 &2 & 2 \end{pmatrix}, \quad
B\in \left\{\begin{pmatrix} 0 \\ 0 \\ 0 \end{pmatrix}, \begin{pmatrix} 1  \\ 0  \\ 0 \end{pmatrix}, \begin{pmatrix} 2 \\ 1 \\ 0\end{pmatrix}, \begin{pmatrix} 3 \\ 2 \\ 1 \end{pmatrix}\right\}.
\end{align*}
Note that the last case was missing in Zagier's list \cite[Table 3]{Zagier}.

\begin{theorem}\label{thm-9}
We have
\begin{align}
\sum_{i,j,k\geq 0}\frac{q^{3i^2+2j^2+k^2+4ij+2ik+2jk}}{(q;q)_i(q;q)_j(q;q)_k}&=\frac{(q^4,q^5,q^9;q^9)_\infty}{(q;q)_\infty}, \label{exam9-1} \\
\sum_{i,j,k\geq 0}\frac{q^{3i^2+2j^2+k^2+4ij+2ik+2jk+i}}{(q;q)_i(q;q)_j(q;q)_k}&=\frac{(q^3,q^6,q^9;q^9)_\infty}{(q;q)_\infty}, \label{exam9-2} \\
\sum_{i,j,k\geq 0}\frac{q^{3i^2+2j^2+k^2+4ij+2ik+2jk+2i+j}}{(q;q)_i(q;q)_j(q;q)_k}&=\frac{(q^2,q^7,q^9;q^9)_\infty}{(q;q)_\infty}, \label{exam9-3} \\
\sum_{i,j,k\geq 0}\frac{q^{3i^2+2j^2+k^2+4ij+2ik+2jk+3i+2j+k}}{(q;q)_i(q;q)_j(q;q)_k}&=\frac{(q,q^8,q^9;q^9)_\infty}{(q;q)_\infty}. \label{exam9-4}
\end{align}
\end{theorem}
\begin{proof}
If we set $k=4$ and $s=4,3,2,1$ in the Andrews-Gordon identity \eqref{AG}, we obtain \eqref{exam9-1}--\eqref{exam9-4}, respectively.
\end{proof}

\subsection{Example 10.}
The matrix and vector parts for this example are
\begin{align*}
A=\begin{pmatrix} 4 & 2 & 2 \\ 2 & 2 & 1 \\ 2 & 1 & 2 \end{pmatrix}, \quad & B\in \Bigg\{ \begin{pmatrix} 0 \\ -1/2 \\ 0 \end{pmatrix}, \begin{pmatrix} 0 \\ 0 \\ -1/2 \end{pmatrix},
\begin{pmatrix}  0 \\ 0 \\ 0 \end{pmatrix}, \begin{pmatrix} 1 \\ 0 \\ 1/2 \end{pmatrix}, \\
& \begin{pmatrix} 1 \\ 0 \\ 1 \end{pmatrix}, \begin{pmatrix} 1 \\ 1/2 \\ 0 \end{pmatrix}, \begin{pmatrix} 1 \\ 1 \\ 0 \end{pmatrix},
\begin{pmatrix} 2 \\ 1 \\ 1 \end{pmatrix} \Bigg\}.
\end{align*}
Since the quadratic form $n^\mathrm{T}A n$ is symmetric in $n_2$ and $n_3$, there are essentially five identities.
\begin{theorem}\label{thm-10}
We have
\begin{align}
\sum_{i,j,k\geq 0} \frac{q^{4i^2+2j^2+2k^2+4ij+4ik+2jk-j}}{(q^2;q^2)_i(q^2;q^2)_j(q^2;q^2)_k} &=\frac{1}{(q,q^4;q^5)_\infty},  \label{exam10-1}\\
\sum_{i,j,k\geq 0} \frac{q^{2i^2+j^2+k^2+2ij+2ik+jk}}{(q;q)_i (q;q)_j (q;q)_k}&=\frac{1}{(q,q^4;q^5)_\infty^2}, \label{exam10-2} \\
\sum_{i,j,k\geq 0} \frac{q^{4i^2+2j^2+2k^2+4ij+4ik+2jk+2i+j}}{(q^2;q^2)_i(q^2;q^2)_j(q^2;q^2)_k}&=\frac{1}{(q^2,q^3;q^5)_\infty}, \label{exam10-3} \\
\sum_{i,j,k\geq 0} \frac{q^{2i^2+j^2+k^2+2ij+2ik+jk+i+j}}{(q;q)_i (q;q)_j (q;q)_k}&=\frac{(q^5;q^5)_\infty}{(q;q)_\infty}, \label{exam10-4} \\
\sum_{i,j,k\geq 0} \frac{q^{2i^2+j^2+k^2+2ij+2ik+jk+2i+j+k}}{(q;q)_i (q;q)_j (q;q)_k}&=\frac{1}{(q^2,q^3;q^5)_\infty^2}. \label{exam10-5}
\end{align}
\end{theorem}
\begin{proof}
We define
\begin{align}
F(u,v,w;q):=\sum_{i,j,k\geq 0} \frac{u^iv^jw^kq^{2i^2+j^2+k^2+2ij+2ik+jk}}{(q;q)_i(q;q)_j(q;q)_k}.
\end{align}
Replacing $q$ by $q^2$, by \eqref{Euler} and \eqref{Jacobi} we have
\begin{align}
F(u,v,w;q^2)&=\sum_{i,j,k\geq 0}\frac{u^iv^jw^k q^{(2i+j+k)^2+j^2+k^2}}{(q^2;q^2)_i(q^2;q^2)_j(q^2;q^2)_k} \nonumber \\
&=\oint \sum_{i\geq 0}\frac{u^i z^{2i}}{(q^2;q^2)_i} \sum_{j\geq 0}\frac{q^{j^2}v^jz^j}{(q^2;q^2)_j} \sum_{k\geq 0}\frac{q^{k^2}w^kz^k}{(q^2;q^2)_k} \sum_{\ell=-\infty}^\infty q^{\ell^2}z^{-\ell} \frac{\mathrm{d}z} {2\pi iz}  \nonumber \\
&=\oint \frac{(-vqz,-wqz,-qz,-q/z,q^2;q^2)_\infty}{(uz^2;q^2)_\infty} \frac{\mathrm{d}z} {2\pi iz}. \label{exam10-F}
\end{align}

(1) By \eqref{exam10-F}, the left side of \eqref{exam10-1} is the same with
\begin{align}
&F(1,q^{-1},1;q^2)=\oint \frac{(-z,-qz,-qz,-q/z,q^2;q^2)_\infty}{(z^2;q^2)_\infty} \frac{\mathrm{d}z} {2\pi iz} \nonumber \\
&=\oint \frac{(-qz,-q/z,q^2;q^2)_\infty}{(z;q)_\infty}  \frac{\mathrm{d}z} {2\pi iz} =\oint \sum_{n=0}^\infty \frac{z^n}{(q;q)_n} \sum_{\ell=-\infty}^\infty z^{-\ell}q^{\ell^2} \frac{\mathrm{d}z} {2\pi iz} \nonumber \\
&=\sum_{n=0}^\infty \frac{q^{n^2}}{(q;q)_n}.
\end{align}
By \eqref{RR-1} we get \eqref{exam10-1}.

(2) After replacing $q$ by $q^2$ and using \eqref{exam10-F}, the left side of \eqref{exam10-2} is the same with
\begin{align}
F(1,1,1;q^2)&=\oint \frac{(-qz,-qz,-qz,-q/z,q^2;q^2)_\infty}{(z^2;q^2)_\infty} \frac{\mathrm{d}z} {2\pi iz} \nonumber \\
&=\oint \frac{(-qz,-qz,-q/z,q^2;q^2)_\infty}{(z,-z,qz;q^2)_\infty} \frac{\mathrm{d}z} {2\pi iz}. \label{exam10-2-F}
\end{align}
Applying Lemma \ref{lem-integral} with $q$ replaced by $q^2$ and
\begin{align*}
(A,B,C,D)=(2,1,3,0), \quad (a_1,a_2)=(-q,-q), \\
 b_1=-q, \quad (c_1,c_2,c_3)=(1,-1,q),
\end{align*}
we deduce that
\begin{align}
F(1,1,1;q^2)=R_1(q)+R_2(q)+R_3(q), \label{exam10-2-split}
\end{align}
where
\begin{align}
R_1(q)&=\frac{(-q,-q,-q;q^2)_\infty}{(-1,q;q^2)_\infty}\sum_{n=0}^\infty \frac{q^{n(n+2)}(-q,-q;q^2)_n}{(q^2,-q,-q^2,q;q^2)_n}, \\
R_2(q)&=\frac{(q,q,q;q^2)_\infty}{(-1,-q;q^2)_\infty}\sum_{n=0}^\infty \frac{(-1)^nq^{n(n+2)}(q,q;q^2)_n}{(q^2,q,-q^2,-q;q^2)_n}, \\
R_3(q)&=\frac{(-q^2,-1,-1;q^2)_\infty}{(-q^{-1},q^{-1};q^2)_\infty} \sum_{n=0}^\infty \frac{q^{n(n+3)}(-q^2,-q^2;q^2)_n}{(q^2,-q^2,-q^3,q^3;q^2)_n}.
\end{align}
Now we evaluate these sums one by one.

We have
\begin{align}
R_1(q)&=\frac{(-q;q^2)_\infty^3}{2(-q^2,q;q^2)_\infty}\sum_{n=0}^\infty \frac{q^{n(n+2)}(-q;q^2)_n}{(q;q^2)_n(q^4;q^4)_n} =\frac{1}{2}\frac{J_2^8J_{10}^2J_{20}^5}{J_1^4J_4^4J_{3,20}J_{4,20}^2J_{5,20}^2J_{6,20}J_{7,20}}. \label{exam10-2-R1-result}
\end{align}
Here for the second equality we used \eqref{eq-BMS-dual}.

Similarly, by \eqref{eq-BMS} we have
\begin{align}
R_2(q)&=\frac{(q;q^2)_\infty^3}{2(-q;q)_\infty}\sum_{n=0}^\infty \frac{(-1)^nq^{n(n+2)}(q;q^2)_n}{(-q;q^2)_n(q^4;q^4)_n} =\frac{1}{2} \frac{J_1^5J_5^2J_{2,10}}{J_2^6J_{10}J_{1,5}}. \label{exam10-2-R2-result}
\end{align}
Next, by \eqref{Slater43} with $q$ replaced by $q^2$, we have
\begin{align}
R_3(q)&=-4q^2\frac{(-q^2;q^2)_\infty^3}{(-q,q;q^2)_\infty}\sum_{n=0}^\infty \frac{q^{n(n+3)}(-q^2;q^2)_n}{(q^2;q^2)_n(q^2;q^4)_{n+1}} =-4q^2\frac{J_4^4J_{20}^3}{J_2^5J_{10}J_{6,20}}. \label{exam10-2-R3-result}
\end{align}
Now substituting \eqref{exam10-2-R1-result}--\eqref{exam10-2-R3-result} into \eqref{exam10-2-split}, using the method in \cite{Garvan-Liang}, it is easy to verify that \eqref{exam10-2} holds.

(3) By \eqref{exam10-F}, the left side of \eqref{exam10-3} is the same with
\begin{align}
&F(q^2,q,1;q^2)=\oint \frac{(-q^2z,-qz,-qz,-q/z,q^2;q^2)_\infty}{(q^2z^2;q^2)_\infty} \frac{\mathrm{d}z} {2\pi iz} \nonumber \\
&=\oint \frac{(-qz,-q/z,q^2;q^2)_\infty}{(qz;q)_\infty} \frac{\mathrm{d}z} {2\pi iz}=\sum_{n=0}^\infty \frac{q^nz^n}{(q;q)_n} \sum_{\ell=-\infty}^\infty z^{-\ell}q^{\ell^2} \frac{\mathrm{d}z} {2\pi iz} \nonumber \\
&=\sum_{n=0}^\infty \frac{q^{n^2+n}}{(q;q)_n}.
\end{align}
By \eqref{RR-2} we get \eqref{exam10-3}.

(4) After replacing $q$ by $q^2$ and using \eqref{exam10-F}, the left side of \eqref{exam10-4} is the same with
\begin{align}
F(q^2,q^2,1;q^2)&=\oint \frac{(-q^3z,-qz,-qz,-q/z,q^2;q^2)_\infty}{(q^2z^2;q^2)_\infty}\frac{\mathrm{d}z}{2\pi iz} \nonumber \\
&=\oint \frac{(-qz,-q^3z,-q/z,q^2;q^2)_\infty}{(qz,q^2z,-q^2z;q^2)_\infty} \frac{\mathrm{d}z}{2\pi iz}. \label{exam10-4-F}
\end{align}
Applying Lemma \ref{lem-integral} with $q$ replaced by $q^2$ and
\begin{align*}
(A,B,C,D)=(2,1,3,0), \quad (a_1,a_2)=(-q,-q^3), \\
b_1=-q,\quad (c_1,c_2,c_3)=(q,q^2,-q^2),
\end{align*}
we deduce that
\begin{align}\label{exam10-4-split}
F(q^2,q^2,1;q^2)=S_1(q)+S_2(q)+S_3(q)
\end{align}
where
\begin{align}
S_1(q)&=\frac{(-q^2,-1,-q^2;q^2)_\infty}{(q,-q;q^2)_\infty}\sum_{n=0}^\infty \frac{q^{n(n+1)}(-q^2,-1;q^2)_n}{(q^2,-q^2,q,-q;q^2)_n}, \\
S_2(q)&=\frac{(-q^3,-q^{-1},-q;q^2)_\infty}{(q^{-1},-1;q^2)_\infty} \sum_{n=0}^\infty \frac{q^{n(n+2)}(-q^3,-q;q^2)_n}{(q^2,-q^2,q^3,-q^3;q^2)_n}, \\
S_3(q)&=\frac{(q^3,q,q^{-1};q^2)_\infty}{(-1,-q^{-1};q^2)_\infty} \sum_{n=0}^\infty \frac{(-1)^nq^{n(n+2)}(q^3,q;q^2)_n}{(q^2,q^3,-q^2,-q^3;q^2)_n}.
\end{align}
We now evaluate these sums one by one.

We have
\begin{align}
S_1(q)=2\frac{(-q^2;q^2)_\infty^3}{(q^2;q^4)_\infty} \sum_{n=0}^\infty \frac{q^{n(n+1)}(-1;q^2)_n}{(q^2;q^2)_n(q^2;q^4)_n}=2\frac{J_4^5J_{10}^2}{J_2^6J_{20}}. \label{exam10-4-S1-result}
\end{align}
where the last equality follows from \eqref{eq-Rogers-mod10} with $q$ replaced by $q^2$.

Next, using \eqref{eq-MSP-dual} and \eqref{eq-MSP} we get
\begin{align}
S_2(q)&=-\frac{(-q;q^2)_\infty^3}{2(-q^2;q^2)_\infty (q;q^2)_\infty} \sum_{n=0}^\infty \frac{q^{n(n+2)}(-q;q^2)_n}{(q;q^2)_{n+1}(q^4;q^4)_n}=-\frac{1}{2}\frac{J_2^9J_5J_{20}}{J_1^5J_4^5J_{10}}, \label{exam10-4-S2-result} \\
S_3(q)&=-\frac{(q;q^2)_\infty^3}{2(-q^2;q^2)_\infty (-q;q^2)_\infty} \sum_{n=0}^\infty \frac{(-1)^nq^{n(n+2)}(q;q^2)_n }{(-q;q^2)_{n+1}(q^4;q^4)_n}=-\frac{1}{2}\frac{J_1^5J_{10}^2}{J_2^6J_5}. \label{exam10-4-S3-result}
\end{align}
Now substituting \eqref{exam10-4-S1-result}--\eqref{exam10-4-S3-result} into \eqref{exam10-4-split}, using the method in \cite{Garvan-Liang}, it is easy to verify that \eqref{exam10-4} holds.

(5) After replacing $q$ by $q^2$ and using \eqref{exam10-F}, the left side of \eqref{exam10-5} is the same with
\begin{align}
F(q^4,q^2,q^2;q^2)&=\oint \frac{(-q^3z,-q^3z,-qz,-q/z,q^2;q^2)_\infty}{(q^4z^2;q^2)_\infty} \frac{\mathrm{d}z}{2\pi iz} \nonumber \\
&=\oint \frac{(-q^3z,-qz,-q/z,q^2;q^2)_\infty}{(q^2z,-q^2z,q^3z;q^2)_\infty} \frac{\mathrm{d}z}{2\pi iz}. \label{exam10-5-F}
\end{align}
Applying Lemma \ref{lem-integral} with $q$ replaced by $q^2$ and
\begin{align*}
(A,B,C,D)=(2,1,3,0), \quad (a_1,a_2)=(-q^3,-q), \\
b_1=-q, \quad (c_1,c_2,c_3)=(q^2,-q^2,q^3),
\end{align*}
we deduce that
\begin{align}
F(q^4,q^2,q^2;q^2)=T_1(q)+T_2(q)+T_3(q), \label{exam10-5-split}
\end{align}
where
\begin{align}
T_1(q)&=\frac{(-q^3,-q,-q^{-1};q^2)_\infty}{(-1,q;q^2)_\infty} \sum_{n=0}^\infty \frac{q^{n^2}(-q,-q^3;q^2)_n}{(q^2,-q^3,-q^2,q;q^2)_n}, \\
T_2(q)&=\frac{(q^3,q,q^{-1};q^2)_\infty}{(-1,-q;q^2)_\infty} \sum_{n=0}^\infty \frac{(-1)^nq^{n^2}(q,q^3;q^2)_n}{(q^2,q^3,-q^2,-q;q^2)_n},\\
T_3(q)&=\frac{(-q^4,-1,-q^{-2};q^2)_\infty}{(-q^{-1},q^{-1};q^2)_\infty} \sum_{n=0}^\infty \frac{q^{n(n+1)}(-q^2,-q^4;q^2)_n}{(q^2,-q^4,-q^3,q^3;q^2)_n}.
\end{align}
Now we evaluate these sums one by one.

By \eqref{Slater21-dual} we have
\begin{align}
T_1(q)&=\frac{1}{2}q^{-1}\frac{(-q;q^2)_\infty^3}{(-q^2,q;q^2)_\infty} \sum_{n=0}^\infty \frac{q^{n^2}(-q;q^2)_n}{(q;q^2)_n(q^4;q^4)_n} \nonumber \\
&=\frac{1}{2}q^{-1}\frac{J_2^8J_{10}^2J_{20}^5}{J_1^4J_4^4J_{1,20}J_{2,20}J_{5,20}^2J_{8,20}^2J_{9,20}}. \label{exam10-5-T1-result}
\end{align}
Similarly, using \eqref{Slater21} we have
\begin{align}
T_2(q)&=-\frac{1}{2}q^{-1}\frac{(q;q^2)_\infty^3}{(-q^2,-q;q^2)_\infty} \sum_{n=0}^\infty \frac{(-1)^nq^{n^2}(q;q^2)_n}{(-q;q^2)_n(q^4;q^4)_n} =-\frac{1}{2}q^{-1}\frac{J_1^4J_5^2J_{1,10}}{J_2^4J_{10}J_{2,10}^2}. \label{exam10-5-T2-result}
\end{align}
We can also get \eqref{exam10-5-T2-result} directly from \eqref{exam10-5-T1-result}  by noting that $T_2(q)=T_1(-q)$.

Next, using \eqref{Slater45} with $q$ replaced by $q^2$ we have
\begin{align}
T_3(q)&=-4\frac{(-q^2;q^2)_\infty^3}{(-q,q;q^2)_\infty} \sum_{n=0}^\infty \frac{q^{n(n+1)}(-q^2;q^2)_n}{(q^2;q^2)_n(q^2;q^4)_{n+1}} =-4\frac{J_4^4J_{20}^3}{J_2^5J_{10}J_{2,20}}. \label{exam10-5-T3-result}
\end{align}
Substituting \eqref{exam10-5-T1-result}--\eqref{exam10-5-T3-result} into \eqref{exam10-5-split}, using the method in \cite{Garvan-Liang}, it is easy to verify that \eqref{exam10-5} holds.
\end{proof}
\begin{rem}\label{rem-Andrews}
The identity \eqref{exam10-2} appeared in Andrews' paper \cite[Eq.\ (4.4)]{Andrews1981HJM}. According to \cite{Andrews1981HJM}, it was first discovered by Andrews through computer search. Then R. Askey provided a proof of it, and Andrews \cite{Andrews1981HJM} proved the following general version:
\begin{align}
\sum_{n_1,\cdots,n_r,v_1,\cdots,v_{r-1}\geq 0} \frac{q^{Q_{2r-1}(n_1,\cdots,n_{r},v_1,\cdots,v_{r-1})}}{(q;q)_{n_1}\cdots (q;q)_{n_r}(q;q)_{v_1}\cdots (q;q)_{v_{r-1}}}=\frac{1}{(q,q^4;q^5)_\infty^r}, \label{eq-Andrews-general}
\end{align}
where
\begin{align*}
&Q_{2r-1}(n_1,\cdots,n_r,v_1,\cdots,v_{r-1})\nonumber \\
&=(n_1+v_1+\cdots+v_{r-1})^2+(n_2+v_1)^2+(n_3+v_2)^2+\cdots+(n_r+v_{r-1})^2\nonumber \\
&+(n_1+v_2+\cdots+v_r)n_2+(n_1+v_3+\cdots+v_r)n_3+\cdots +n_1n_r.
\end{align*}
Our proof here is quite different from the proofs given by Askey and Andrews.
\end{rem}

Using the integral representations in \eqref{exam10-2-F}, \eqref{exam10-4-F} and \eqref{exam10-5-F}, we get two equivalent identities for each of \eqref{exam10-2}, \eqref{exam10-4} and \eqref{exam10-5}.
\begin{corollary}\label{cor-10}
We have
\begin{align}
\sum_{i,j,k\geq 0} \frac{q^{\frac{1}{2}i^2+j^2+2k^2+ij+2ik+2jk+\frac{1}{2}i}}{(q;q)_i(q;q)_j(q^2;q^2)_k}&=\frac{1}{(q,q^4;q^5)_\infty^2}, \label{cor-10-1} \\
\sum_{i,j,k\geq 0} \frac{(-1)^jq^{i^2+j^2+2k^2+2ij+2ik+2jk}}{(q;q)_i(q^2;q^2)_j(q^2;q^2)_k}&=\frac{1}{(q^2,q^8;q^{10})_\infty^2}, \label{cor-10-2} \\
\sum_{i,j,k\geq 0} \frac{q^{\frac{1}{2}i^2+j^2+2k^2+ij+2ik+2jk+\frac{1}{2}i+j+2k}}{(q;q)_i(q;q)_j(q^2;q^2)_k}&=\frac{(q^5;q^5)_\infty}{(q;q)_\infty},\label{cor-10-3} \\
\sum_{i,j,k\geq 0} \frac{(-1)^jq^{i^2+j^2+2k^2+2ij+2ik+2jk+i+2j+2k}}{(q;q)_i(q^2;q^2)_j(q^2;q^2)_k}&=\frac{(q^{10};q^{10})_\infty}{(q^2;q^2)_\infty}, \label{cor-10-4} \\
\sum_{i,j,k\geq 0} \frac{q^{\frac{1}{2}i^2+j^2+2k^2+ij+2ik+2jk+\frac{3}{2}i+j+2k}}{(q;q)_i(q;q)_j(q^2;q^2)_k}&=\frac{1}{(q^2,q^3;q^5)_\infty^2}, \label{cor-10-5}\\
\sum_{i,j,k\geq 0} \frac{(-1)^jq^{i^2+j^2+2k^2+2ij+2ik+2jk+2i+2j+2k}}{(q;q)_i(q^2;q^2)_j(q^2;q^2)_k}&=\frac{1}{(q^4,q^6;q^{10})_\infty^2}. \label{cor-10-6}
\end{align}
\end{corollary}
\begin{proof}
By \eqref{exam10-2-F}, \eqref{Euler} and \eqref{Jacobi} we have
\begin{align}
F(1,1,1;q^2)&=\oint \frac{(-qz,-qz,-q/z,q^2;q^2)_\infty}{(qz;q^2)_\infty (z^2;q^4)_\infty} \frac{\mathrm{d}z}{2\pi iz} \nonumber \\
&=\oint \sum_{i\geq 0} \frac{q^iz^i}{(q^2;q^2)_i} \sum_{j\geq 0} \frac{q^{j^2}z^j}{(q^2;q^2)_j} \sum_{k\geq 0} \frac{z^{2k}}{(q^4;q^4)_k} \sum_{\ell=-\infty}^\infty z^{-\ell}q^{\ell^2} \frac{\mathrm{d}z}{2\pi iz} \nonumber \\
&=\sum_{i,j,k\geq 0} \frac{q^{i+j^2+(i+j+2k)^2}}{(q^2;q^2)_i(q^2;q^2)_j(q^4;q^4)_k}. \label{cor-10-1-proof}
\end{align}
Since we have proved that
\begin{align}\label{cor-10-12-id}
F(1,1,1;q)=\frac{1}{(q,q^4;q^5)_\infty^2}.
\end{align}
Replacing $q^2$ by $q$ in \eqref{cor-10-1-proof}, we obtain \eqref{cor-10-1}.

Now we use \eqref{exam10-2-F} in a different way. We have
\begin{align}
F(1,1,1;q^2)&=\oint \frac{(-qz,-qz,-q/z,q^2;q^2)_\infty}{(z;q)_\infty (-z;q^2)_\infty} \frac{\mathrm{d}z}{2\pi iz} \nonumber \\
&=\oint \sum_{i\geq 0} \frac{z^i}{(q;q)_i} \sum_{j\geq 0} \frac{(-1)^jz^j}{(q^2;q^2)_j} \sum_{k\geq 0} \frac{z^{k}q^{k^2}}{(q^2;q^2)_k}  \sum_{\ell=-\infty}^\infty z^{-\ell}q^{\ell^2} \frac{\mathrm{d}z}{2\pi iz} \nonumber \\
&=\sum_{i,j,k\geq 0} \frac{(-1)^jq^{(i+j+k)^2+k^2}}{(q;q)_i(q^2;q^2)_j(q^2;q^2)_k}.
\end{align}
Now using \eqref{cor-10-12-id} with $q$ replaced by $q^2$, we obtain \eqref{cor-10-2}.

In the same way, using \eqref{exam10-4-F} we can prove \eqref{cor-10-3} and \eqref{cor-10-4}. Using \eqref{exam10-5-F} we can prove \eqref{cor-10-5} and \eqref{cor-10-6}.
\end{proof}

\subsection{Example 11.}
The matrix and vector parts for this example are
$$A=\begin{pmatrix} 4 & 2 & -1 \\ 2 & 2 & -1 \\ -1 & -1 & 1 \end{pmatrix}, \quad
B\in \left\{\begin{pmatrix} 0 \\ 0 \\ 0 \end{pmatrix}, \begin{pmatrix} 0 \\ 0 \\ 1/2 \end{pmatrix},
\begin{pmatrix} 1 \\ 0 \\ 0 \end{pmatrix}, \begin{pmatrix} 2 \\ 1 \\ -1/2 \end{pmatrix},
\begin{pmatrix} 2 \\ 1 \\ 0 \end{pmatrix}\right\}.$$

\begin{theorem}\label{thm-11}
We have
\begin{align}
\sum_{i,j,k\geq 0}\frac{q^{4i^2+2j^2+k^2+4ij-2ik-2jk}}{(q^2;q^2)_i(q^2;q^2)_j(q^2;q^2)_k}&=\frac{J_2^3J_{5,12}}{J_1^2J_4^2}, \label{exam11-1} \\
\sum_{i,j,k\geq 0} \frac{q^{2i^2+j^2+\frac{1}{2}k^2+2ij-ik-jk+\frac{1}{2}k}}{(q;q)_i(q;q)_j(q;q)_k}&=\frac{J_2^2J_3^2}{J_1^3J_{6}}, \label{exam11-2} \\
\sum_{i,j,k\geq 0} \frac{q^{4i^2+2j^2+k^2+4ij-2ik-2jk+2i}}{(q^2;q^2)_i(q^2;q^2)_j(q^2;q^2)_k}&=\frac{J_2^3J_{3,12}}{J_1^2J_4^2}, \label{exam11-3}\\
\sum_{i,j,k\geq 0} \frac{q^{2i^2+j^2+\frac{1}{2}k^2+2ij-ik-jk+2i+j-\frac{1}{2}k}}{(q;q)_i(q;q)_j(q;q)_k}&=2\frac{J_2J_6^2}{J_1^2J_3}, \label{exam11-4} \\
\sum_{i,j,k\geq 0} \frac{q^{4i^2+2j^2+k^2+4ij-2ik-2jk+4i+2j}}{(q^2;q^2)_i(q^2;q^2)_j(q^2;q^2)_k}&=\frac{J_2^3J_{1,12}}{J_1^2J_4^2}. \label{exam11-5}
\end{align}
\end{theorem}
\begin{proof}
(1) We have
\begin{align}
&\sum_{i,j,k\geq 0}\frac{q^{4i^2+2j^2+k^2+4ij-2ik-2jk}}{(q^2;q^2)_i(q^2;q^2)_j(q^2;q^2)_k} \nonumber \\
&=\sum_{i,j\geq 0}\frac{q^{4i^2+2j^2+4ij}}{(q^2;q^2)_i(q^2;q^2)_j}\sum_{k\geq 0}\frac{q^{k^2-k}\cdot q^{(1-2i-2j)k}}{(q^2;q^2)_k} \nonumber \\
&=\sum_{i,j\geq 0}\frac{q^{2i^2+2(i+j)^2}}{(q^2;q^2)_i(q^2;q^2)_j}(-q^{1-2i-2j};q^2)_\infty  \quad \text{(set $n=i+j$)}\nonumber \\
&=\sum_{n=0}^\infty \sum_{i=0}^n \frac{q^{2n^2+2i^2}}{(q^2;q^2)_i(q^2;q^2)_{n-i}}(-q^{1-2n};q^2)_\infty  \nonumber \\
&=(-q;q^2)_\infty \sum_{n=0}^\infty \frac{q^{n^2}(-q;q^2)_n}{(q^2;q^2)_n}\sum_{i=0}^n q^{2i^2} {n\brack i}_{q^2} .
\end{align}
Now by \eqref{eq-thm-B1} we get \eqref{exam11-1}.

(2) We have
\begin{align}
&\sum_{i,j,k\geq 0} \frac{q^{2i^2+j^2+\frac{1}{2}k^2+2ij-ik-jk+\frac{1}{2}k}}{(q;q)_i(q;q)_j(q;q)_k}\nonumber \\
&=\sum_{i,j\geq 0} \frac{q^{2i^2+j^2+2ij}}{(q;q)_i(q;q)_j}\sum_{k\geq 0} \frac{q^{(k^2-k)/2}\cdot (q^{1-i-j})^k}{(q;q)_k} \nonumber \\
&=\sum_{i,j\geq 0}\frac{q^{i^2+(i+j)^2}}{(q;q)_i(q;q)_j} (-q^{1-i-j};q)_\infty \quad \text{(set $n=i+j$)} \nonumber \\
&=\sum_{n=0}^\infty \sum_{i=0}^n \frac{q^{n^2+i^2}}{(q;q)_i(q;q)_{n-i}} (-q^{1-n};q)_\infty \nonumber \\
&=(-q;q)_\infty \sum_{n=0}^\infty \frac{q^{n(n+1)/2}(-1;q)_n}{(q;q)_n} \sum_{i=0}^n q^{i^2} {n\brack i}.
\end{align}
Now by \eqref{eq-thm-B2} we get \eqref{exam11-2}.

(3) We have
\begin{align}
&\sum_{i,j,k\geq 0} \frac{q^{4i^2+2j^2+k^2+4ij-2ik-2jk+2i}}{(q^2;q^2)_i(q^2;q^2)_j(q^2;q^2)_k} \nonumber \\
&=\sum_{i,j\geq 0} \frac{q^{4i^2+2j^2+4ij+2i}}{(q^2;q^2)_i(q^2;q^2)_j} \sum_{k\geq 0}\frac{q^{k^2-k}\cdot q^{(1-2i-2j)k}}{(q^2;q^2)_k} \nonumber \\
&=\sum_{i,j\geq 0}\frac{q^{2i^2+2i+2(i+j)^2}}{(q^2;q^2)_i(q^2;q^2)_j}(-q^{1-2i-2j};q^2)_\infty  \quad \text{(set $n=i+j$)} \nonumber \\
&=\sum_{n=0}^\infty \sum_{i=0}^n \frac{q^{2n^2+2i^2+2i}}{(q^2;q^2)_i(q^2;q^2)_{n-i}}(-q^{1-2n};q^2)_\infty \nonumber \\
&=(-q;q^2)_\infty \sum_{n=0}^\infty  \frac{q^{n^2}(-q;q^2)_n}{(q^2;q^2)_n} \sum_{i=0}^n q^{2i^2+2i} {n \brack i}_{q^2}.
\end{align}
Now by \eqref{eq-thm-B3} we get \eqref{exam11-3}.

(4) We have
\begin{align}
&\sum_{i,j,k\geq 0} \frac{q^{2i^2+j^2+\frac{1}{2}k^2+2ij-ik-jk+2i+j-\frac{1}{2}k}}{(q;q)_i(q;q)_j(q;q)_k}\nonumber \\
&=\sum_{i,j\geq 0}\frac{q^{2i^2+j^2+2ij+2i+j}}{(q;q)_i(q;q)_j} \sum_{k\geq 0}\frac{q^{(k^2-k)/2}\cdot q^{-(i+j)k}}{(q;q)_k} \nonumber \\
&=\sum_{i,j\geq 0}\frac{q^{i^2+i+(i+j)^2+i+j}}{(q;q)_i(q;q)_j} (-q^{-i-j};q)_\infty \quad \text{(set $n=i+j$)} \nonumber \\
&=\sum_{n=0}^\infty \sum_{i=0}^n \frac{q^{n^2+n+i^2+i}}{(q;q)_i(q;q)_{n-i}}(-q^{-n};q)_\infty \nonumber \\
&=(-1;q)_\infty \sum_{n=0}^\infty \frac{q^{(n^2+n)/2}(-q;q)_n}{(q;q)_n}\sum_{i=0}^n q^{i^2+i} { n\brack i}.
\end{align}
Now by \eqref{eq-thm-B4} we get \eqref{exam11-4}.

(5) We have
\begin{align}
&\sum_{i,j,k\geq 0} \frac{q^{4i^2+2j^2+k^2+4ij-2ik-2jk+4i+2j}}{(q^2;q^2)_i(q^2;q^2)_j(q^2;q^2)_k} \nonumber \\
&=\sum_{i,j\geq 0}\frac{q^{4i^2+2j^2+4ij+4i+2j}}{(q^2;q^2)_i(q^2;q^2)_j} \sum_{k\geq 0} \frac{q^{k^2-k}\cdot q^{(1-2i-2j)k}}{(q^2;q^2)_k} \nonumber \\
&=\sum_{i,j\geq 0}\frac{q^{2i^2+2i+2(i+j)^2+2i+2j}}{(q^2;q^2)_i(q^2;q^2)_j} (-q^{1-2i-2j};q^2)_\infty  \quad \text{(set $n=i+j$)} \nonumber \\
&=\sum_{n=0}^\infty \sum_{i=0}^n \frac{q^{2n^2+2n+2i^2+2i}}{(q^2;q^2)_i(q^2;q^2)_{n-i}} (-q^{1-2n};q^2)_\infty \nonumber \\
&=(-q;q^2)_\infty \sum_{n=0}^\infty \frac{q^{n^2+2n}(-q;q^2)_n}{(q^2;q^2)_n} \sum_{i=0}^n q^{2i^2+2i} {n\brack i}_{q^2}.
\end{align}
Now by \eqref{eq-thm-B5} we get \eqref{exam11-5}.
\end{proof}

\subsection{Example 12.}
The matrix and vector parts for this example are
$$A=\begin{pmatrix} 8 & 4 & 1 \\ 4 & 3 & 0 \\ 1 & 0 &1 \end{pmatrix}, \quad
B\in \left\{\begin{pmatrix} 0 \\ -1/2 \\ 1/2 \end{pmatrix}, \begin{pmatrix} 2 \\ 1/2 \\ 1/2 \end{pmatrix} \right\}.$$

\begin{theorem}\label{thm-12}
We have
\begin{align}
\sum_{i,j,k\geq 0} \frac{q^{4i^2+\frac{3}{2}j^2+\frac{1}{2}k^2+4ij+ik-\frac{1}{2}j+\frac{1}{2}k}}{(q;q)_i(q;q)_j(q;q)_k}
&=\frac{(-q;q)_\infty}{(q,q^4;q^5)_\infty}, \label{exam12-1} \\
\sum_{i,j,k\geq 0} \frac{q^{4i^2+\frac{3}{2}j^2+\frac{1}{2}k^2+4ij+ik+2i+\frac{1}{2}j+\frac{1}{2}k}}{(q;q)_i(q;q)_j(q;q)_k}
&=\frac{(-q;q)_\infty}{(q^2,q^3;q^5)_\infty}. \label{exam12-2}
\end{align}
\end{theorem}
\begin{proof}
We have
\begin{align}
&\sum_{i,j,k\geq 0} \frac{q^{4i^2+\frac{3}{2}j^2+\frac{1}{2}k^2+4ij+ik-\frac{1}{2}j+\frac{1}{2}k}}{(q;q)_i(q;q)_j(q;q)_k}=\sum_{i,j\geq 0} \frac{q^{4i^2+\frac{3}{2}j^2+4ij-\frac{1}{2}j}}{(q;q)_i(q;q)_j} \sum_{k\geq 0}\frac{q^{\frac{1}{2}(k^2-k)+(i+1)k}}{(q;q)_k} \nonumber \\
&=\sum_{i,j\geq 0} \frac{q^{4i^2+4ij+\frac{j(3j-1)}{2}}}{(q;q)_i(q;q)_j}(-q^{i+1};q)_\infty  =(-q;q)_\infty \sum_{i,j\geq 0}\frac{q^{4i^2+4ij+\frac{j(3j-1)}{2}}}{(q^2;q^2)_i (q;q)_j}.
\end{align}
Using \eqref{eq-lem-12-1}  (with $i,j$ interchanged)  we get \eqref{exam12-1}.

Similarly,
\begin{align}
&\sum_{i,j,k\geq 0} \frac{q^{4i^2+\frac{3}{2}j^2+\frac{1}{2}k^2+4ij+ik+2i+\frac{1}{2}j+\frac{1}{2}k}}{(q;q)_i(q;q)_j(q;q)_k}
=\sum_{i,j\geq 0} \frac{q^{4i^2+2i+\frac{j(3j+1)}{2}+4ij}}{(q;q)_i(q;q)_j}\sum_{k\geq 0}\frac{q^{\frac{1}{2}(k^2-k)+(i+1)k}}{(q;q)_k} \nonumber \\
&=\sum_{i,j\geq 0} \frac{q^{4i^2+2i+\frac{j(3j+1)}{2}+4ij}}{(q;q)_i(q;q)_j} (-q^{i+1};q)_\infty =(-q;q)_\infty \sum_{i,j\geq 0}\frac{q^{4i^2+2i+4ij+\frac{j(3j+1)}{2}}}{(q^2;q^2)_i(q;q)_j}.
\end{align}
Using \eqref{eq-lem-12-2} (with $i,j$ interchanged) we get \eqref{exam12-2}.
\end{proof}

We end this paper with the following remark. On page 50 of \cite{Zagier}, Zagier also gave a short discussion on the duality of modular triples. Specifically, if $(A,B,C)$ is a rank $r$ modular triple, then it is likely that
$$(A^\star, B^\star, C^\star)=(A^{-1},A^{-1}B,\frac{1}{2}B^\mathrm{T} A^{-1}B-\frac{r}{24}-C)$$
is also a modular triple. In an undergoing work, we have verified some of these dual triples generated by the modular triples in this paper. This will be discussed in a separate paper.

\subsection*{Acknowledgements}
We thank S.O. Warnaar for some helpful comments, especially for telling us that Theorem \ref{thm-Bressoud-new} and the Bailey pairs \eqref{pair-alpha-2}--\eqref{pair-beta-2} appeared in \cite{ASW} and \cite{Warnaar}. We are grateful to A. Milas for sending us an updated version of \cite{Li-Milas}, from which we know that \eqref{eq-Li-Milas} appeared in Andrews' paper \cite{Andrews1981HJM}. This work was supported by the National Natural Science Foundation of China (12171375).


\begin{thebibliography}{0}
\bibitem{Andrews1974} G.E. Andrews, On the General Rogers-Ramanujan Theorem. Providence, RI: Amer. Math. Soc., 1974.

\bibitem{AndrewsSIAM} G.E. Andrews, Problem 74-12, SIAM Review 16 (1974), 390.

\bibitem{Andrews1981HJM} G.E. Andrews. Multiple $q$-series identities, Houston J. Math. 3(1981), 1--13.

\bibitem{Andrews1986}  G.E. Andrews, $q$-series: Their development and application in analysis, number theory, combinatorics, physics, and computer algebra. In: CBMS Regional Conference Series in Mathematics, vol. 66. Amer. Math. Soc, Providence, RI, 1986.

\bibitem{AndrewsPJM} G.E. Andrews, Multiple series Rogers-Ramanujan type identities, Pacific J. Math 114 (1984), 267--283.

\bibitem{Andrews} G.E. Andrews, The Theory of Partitions, Addison-Wesley, 1976; Reissued Cambridge, 1998.

\bibitem{Lost2} G.E. Andrews and B.C. Berndt, Ramanujan's Lost Notebook, Part II, Springer 2009.

\bibitem{ASW} G.E. Andrews, A. Schilling and S.O. Warnaar, An $A_2$ Bailey lemma and Rogers--Ramanujan-type identities, J. Amer. Math. Soc. 12(3) (1999), 677--702.

\bibitem{Andrews-Uncu} G.E. Andrews and A.K. Uncu, Sequences in overpartitions, arXiv:2111.15003v1.

\bibitem{BMS}  D. Bowman, J. Mc Laughlin and A.V. Sills, Some more identities of Rogers-Ramanujan type, Ramanujan J. 18(3) (2009),  307--325.

\bibitem{Bressoud1979} D.M. Bressoud, A generalization of the Rogers-Ramanujan identities for all moduli, J. Combin. Theory Ser. A 27 (1979), 64--68.

\bibitem{Bressoud} D.M. Bressoud, Some identities for terminating $q$-series, Math. Proc. Cambridge Philos. Soc. 81 (1981), 211--223.

\bibitem{CMP} C. Calinescu, A. Milas and M. Penn, Vertex algebraic structure of principal subspaces of basic $A_{2n}^{(2)}$-modules, J. Pure Appl. Algebra 220 (2016), 1752--1784.

\bibitem{Cao-Wang} Z. Cao and L. Wang, Multi-sum Rogers-Ramanujan type identities, arXiv:2205.12786.

\bibitem{Feigin} I. Cherednik and B. Feigin, Rogers-Ramanujan type identities and Nil-DAHA, Adv. Math. 248 (2013), 1050--1088.

\bibitem{Chern} S. Chern, Asymmetric Rogers-Ramanujan type identities. I, The Andrews-Uncu Conjecture, 	arXiv:2203.15168.

\bibitem{Dousse-Lovejoy} J. Dousse and J. Lovejoy, Generalizations of Capparelli's identity, Bull. London Math. Soc. 51 (2019), 193--206.

\bibitem{Dyson} F.J. Dyson, Three identities in Combinatory Analysis, J. London Math. Soc. 18 (1943), 35--39.

\bibitem{Garvan-Liang} F.G. Garvan and J. Liang, Automatic proof of theta-function identities, arXiv:1807.08051.

\bibitem{GR-book} G. Gasper and M. Rahman, Basic Hypergeometric Series, 2nd Edition, Encyclopedia of Mathematics and Its Applications, Vol.\ 96, Cambridge University Press, 2004.

\bibitem{Gordon1961} B. Gordon, A combinatorial generalization of the Rogers-Ramanujan identities, Amer. J. Math. 83 (1961), 393--399.

\bibitem{Hardy} G.H. Hardy, Lectures by Godfrey H. Hardy on the mathematical work of
Ramanujan, Edwards Brothers, Ann Arbor, Michigan, 1937, Fall Term 1936. Notes taken by Marshall Hall at the Institute For Advanced Study, Princeton, NJ.

\bibitem{KR-2015} S. Kanade and M.C. Russell, IdentityFinder and some new identities of Rogers-Ramanujan type, Exp. Math. 24 (2015), no. 4, 419--423.

\bibitem{KR-2019} S. Kanade and M.C. Russell, Staircases to analytic sum-sides for many new integer partition identities of Rogers-Ramanujan type, Electron. J. Combin. 26 (2019), 1--6.

\bibitem{Li-Milas} H. Li and A. Milas, Jet schemes, quantum dilogarithm and Feigin-Stoyanovsky's principal subspaces, arXiv: 2010.02143v1.

\bibitem{Laughlin} J. Mc Laughlin, Some more identities of Kanade-Russell type derived using Rosengren's method, Ann. Comb., https://doi.org/10.1007/s00026-022-00586-3.

\bibitem{MSP} J. Mc Laughlin, A.V. Sills and P. Zimmer, Rogers-Ramanujan computer searches, J. Symbolic Comput. 44(8) (2009), 1068--1078.

\bibitem{LeeThesis} C.-H. Lee, Algebraic structures in modular $q$-hypergeometric series, PhD Thesis, University of California, Berkeley, 2012.

\bibitem{Nahm1994} W. Nahm, Conformal field theory and the dilogarithm, In 11th International Conference on Mathematical Physics (ICMP-11) (Satelite colloquia: New Problems in General Theory of Fields and Particles), Paris, 1994, 662--667.

\bibitem{Nahmconf} W. Nahm, Conformal field theory, dilogarithms and three dimensional manifold, in ``Interface between Physics and Mathematics (Proceedings, Conference in Hangzhou, People's Republic of China, September 1993)'', eds. W. Nahm and J.-M. Shen, World Scientific, Singapore, 1994, 154--165.

\bibitem{Nahm2007} W. Nahm, Conformal field theory and torsion elements of the Bloch group, in ``Frontiers in Number Theory, Physics and Geometry'', II, Springer, 2007, 67--132.

\bibitem{Paule1994} P. Paule, Short and easy computer proofs of the Rogers-Ramanujan identities and identities of similar type, Electronic J. Combin. 1 (1994), 9 pp., \#R10.

\bibitem{Rogers1894} L.J. Rogers, Second memoir on the expansion of certain infinite products, Proc.
London Math. Soc. 25 (1894), 318--343.

\bibitem{Rogers1917} L.J. Rogers, On two theorems of combinatory analysis and some allied identities,
Proc. London Math. Soc. 16 (1917), 315--336.

\bibitem{Rosengren} H. Rosengren, Proofs of some partition identities conjectured by Kanade and Russell, Ramanujan J. https://doi.org/10.1007/s11139-021-00389-9

\bibitem{Schur} I. Schur, Ein Beitrag zur additiven Zahlentheorie und zur Theorie
der kettenbr\"uche, S.–B. Preuss. Akad. Wiss. Phys. Math. Klasse
(1917), 302--321.

\bibitem{Selberg} A. Selberg,  \"Uber einige arithmetische Identit\"aten. Avh. Norske Vid.-Akad. Oslo I, No. 8, 1--23, 1936.

\bibitem{Sills-book} A.V. Sills, An Invitation to the Rogers-Ramanujan Identities, CRC Press (2018).

\bibitem{Slater} L.J. Slater, Further identities of the Rogers-Ramanujan type, Proc. Lond. Math. Soc. (2) 54 (1)
(1952), 147--167.

\bibitem{Stanton} D.\ Stanton, The Bailey-Rogers-Ramanujan group, in $q$-Series, with Applications to Combinatorics, Number Theory, and Physics, B.C. Berndt and K. Ono, eds., Contemp. Math. No. 291, American Mathematical Society, Providence, RI, (2001), pp. 55--70.

\bibitem{Terhoeven} M. Terhoeven,  Rationale konforme Feldtheorien, der Dilogarithmus und Invarianten von 3-Mannigfaltigkeiten, PhD Thesis Universitat Bonn, 1995.

\bibitem{VZ}  M. Vlasenko and S. Zwegers, Nahm's conjecture: asymptotic computations and counterexamples, Commu. Math. Phy. 5(3) (2011), 617--642.

\bibitem{Wang} L. Wang, New proofs of some double sum Rogers-Ramanujan type identities,  Ramanujan J. https://doi.org/10.1007/s11139-022-00654-5

\bibitem{Wang2022} L. Wang, Identities on Zagier's rank two examples for Nahm's conjecture, arXiv:2210.10748v1.

\bibitem{Warnaar} S.O. Warnaar, 50 Years of Bailey's Lemma. In: A. Betten, A. Kohnert, R. Laue, A. Wassermann (eds) Algebraic Combinatorics and Applications. Springer, Berlin, Heidelberg, 2001.

\bibitem{Zagier} D. Zagier, The dilogarithm function, in Frontiers in Number Theory, Physics and Geometry, II, Springer, 2007, 3--65.
\end{thebibliography}
\end{document}